\documentclass[11pt]{amsart}
\usepackage{tikz,amsthm,amsmath,amstext,amssymb,amscd,epsfig,euscript, mathrsfs, dsfont,pspicture,multicol,graphpap,graphics,graphicx,times,enumerate,subfig,sidecap,wrapfig,color}

\usepackage{amsmath}
\usepackage{amssymb}
\usepackage{pgf}
\usepackage{stackrel}
\usepackage[colorlinks,citecolor=red]{hyperref}

\hypersetup{colorlinks=true, linkcolor=black}

\usepackage{enumitem}
\makeatletter
\def\namedlabel#1#2{\begingroup
    #2%
    \def\@currentlabel{#2}%
    \phantomsection\label{#1}\endgroup
}
\makeatother

\def\Xint#1{\mathchoice
   {\XXint\displaystyle\textstyle{#1}}%
   {\XXint\textstyle\scriptstyle{#1}}%
   {\XXint\scriptstyle\scriptscriptstyle{#1}}%
   {\XXint\scriptscriptstyle\scriptscriptstyle{#1}}%
   \!\int}
\def\XXint#1#2#3{{\setbox0=\hbox{$#1{#2#3}{\int}$}
     \vcenter{\hbox{$#2#3$}}\kern-.5\wd0}}

\newtheorem{theorem}{Theorem}[section]
\newtheorem*{theorem*}{Theorem}
\newtheorem{lemma}[theorem]{Lemma}
\newtheorem*{lemma*}{Lemma}
\newtheorem{proposition}[theorem]{Proposition}

\theoremstyle{definition}

\theoremstyle{remark}

\newtheorem{question*}[theorem]{Question}

\numberwithin{equation}{section}

\theoremstyle{theorem}
\newtheorem{ltheorem}{Theorem}

\newcommand{\supp}{ \operatorname{supp}}
\newcommand{\dist}{ \operatorname{dist}}
\newcommand{\diam}{ \operatorname{diam}}
\newcommand{\vv}{\vspace{2mm}}

\newcommand{\om}{{\Omega}}
\newcommand{\vphi}{{\varphi}}
\def\loc{\textrm{loc}}

\newcommand{\R}{{\mathbb R}}       
\newcommand{\DD}{{\mathcal D}}

\newcommand{\HH}{{\mathcal H}}

\newcommand{\LL}{{\mathcal L}}

\newcommand{\AZ}{{\mathcal A}}

\newcommand{\RR}{{\mathcal R}}

\newcommand{\EE}{{\mathcal E}}

\newcommand{\ve}{{\varepsilon}}

\newcommand{\wt}[1]{{\widetilde{#1}}}

\def\Xint#1{\mathchoice		
{\XXint\displaystyle\textstyle{#1}}%
{\XXint\textstyle\scriptstyle{#1}}%
{\XXint\scriptstyle\scriptscriptstyle{#1}}%
{\XXint\scriptscriptstyle\scriptscriptstyle{#1}}%
\!\int}
\def\XXint#1#2#3{{\setbox0=\hbox{$#1{#2#3}{\int}$ }
\vcenter{\hbox{$#2#3$ }}\kern-.58\wd0}}

\def\avint{\;\Xint-}

\makeatletter
\def\@tocline#1#2#3#4#5#6#7{\relax
  \ifnum #1>\c@tocdepth 
  \else
    \par \addpenalty\@secpenalty\addvspace{#2}%
    \begingroup \hyphenpenalty\@M
    \@ifempty{#4}{%
      \@tempdima\csname r@tocindent\number#1\endcsname\relax
    }{%
      \@tempdima#4\relax
    }%
    \parindent\z@ \leftskip#3\relax \advance\leftskip\@tempdima\relax
    \rightskip\@pnumwidth plus4em \parfillskip-\@pnumwidth
    #5\leavevmode\hskip-\@tempdima
      \ifcase #1
       \or\or \hskip 1em \or \hskip 2em \else \hskip 3em \fi%
      #6\nobreak\relax
    \dotfill\hbox to\@pnumwidth{\@tocpagenum{#7}}\par
    \nobreak
    \endgroup
  \fi}
\makeatother

\title[Gradients of single layer potentials]
{Failure of $L^2$ boundedness of gradients of single layer potentials for measures with zero low density}

 \makeindex

\begin{document}

\begin{abstract}
Consider a totally irregular measure $\mu$ in $\R^{n+1}$, that is, the upper density $\limsup_{r\to0}\frac{\mu(B(x,r))}{(2r)^n}$  is positive $\mu$-a.e.\ in $\R^{n+1}$, and the lower density $\liminf_{r\to0}\frac{\mu(B(x,r))}{(2r)^n}$ vanishes $\mu$-a.e. in $\R^{n+1}$. We show that if $T_\mu f(x)=\int K(x,y)\,f(y)\,d\mu(y)$
is an operator whose kernel $K(\cdot,\cdot)$ is the gradient of the fundamental solution for a uniformly elliptic operator in divergence form associated with a matrix with H\"older continuous coefficients, then $T_\mu$ is not bounded in $L^2(\mu)$.
This extends a celebrated result proved previously by Eiderman, Nazarov and Volberg for the $n$-dimensional Riesz transform.
\end{abstract}

\author[Conde-Alonso]{Jos\'e M. Conde-Alonso}

\address{Jos\'e M. Conde-Alonso\\
Department of Mathematics, Brown University, 151 Kassar House, Providence, RI, USA.
}
\email{jconde@math.brown.edu}

\author[Mourgoglou]{Mihalis Mourgoglou}

\address{Mihalis Mourgoglou\\
Departamento de Matem\`aticas, Universidad del Pa\' is Vasco, Aptdo. 644, 48080 Bilbao, Spain and\\
Ikerbasque, Basque Foundation for Science, Bilbao, Spain.
}
\email{michail.mourgoglou@ehu.eus}

\author[Tolsa]{Xavier Tolsa}
\address{Xavier Tolsa
\\
ICREA, Passeig Llu\'{\i}s Companys 23 08010 Barcelona, Catalonia, and\\
Departament de Matem\`atiques and BGSMath
\\
Universitat Aut\`onoma de Barcelona
\\
08193 Bellaterra (Barcelona), Catalonia.
}
\email{xtolsa@mat.uab.cat}

\thanks{J.C. was partially supported by the ERC grant 320501 of the European Research Council (FP7/2007-2013). M.M. was supported  by IKERBASQUE and partially supported by the grant IT-641-13 (Basque Government). X.T. was supported by the ERC grant 320501 of the European Research Council and partially supported by MTM-2016-77635-P,  MDM-2014-044 (MICINN, Spain), by 2017-SGR-395 (Catalonia), and by Marie Curie ITN MAnET (FP7-607647).
}

\maketitle


\tableofcontents

\allowdisplaybreaks

\section{Introduction}
The goal of this work is to extend some of the results for the Riesz transforms by Eiderman, Nazarov, and Volberg \cite{ENV} to other  integral operators associated with elliptic operators in divergence form
with H\"older continuous coefficients. In \cite{ENV} the authors show that, given $n\leq s<n+1$, if $\mu$ is a non-zero Borel measure in $\R^{n+1}$ such that the upper $s$-dimensional density
$$\Theta^{s,*}(x,\mu) = \limsup_{r\to0}\frac{\mu(B(x,r))}{(2r)^s}$$
is positive $\mu$-a.e.\ and the lower $s$-dimensional density
$$\Theta^{s}_*(x,\mu) = \liminf_{r\to0}\frac{\mu(B(x,r))}{(2r)^s}$$
vanishes $\mu$-a.e., then the $s$-dimensional Riesz transform $\RR_\mu^s$ cannot be bounded in $L^2(\mu)$. Recall that the $s$-dimensional Riesz transform $\RR_\mu^s$ is defined by
$$\RR_\mu^s f(x) = \int \frac{x-y}{|x-y|^{s+1}}\,f(y)\,d\mu(y),$$
whenever the integral makes sense.

In particular, in the case $s=n$, the results in \cite{ENV} imply that if $E\subset \R^{n+1}$ is a set with positive and finite Hausdorff measure $\HH^n$  
such that $\Theta^{n}_*(x,\HH^n|_E) =0$ for $\HH^n$-a.e.\ $x \in E$,
then $\RR_{\HH^n|_E}^n$ cannot be bounded on $L^2(\HH^n|_E)$.
Let us remark that the fact that $\Theta^{n}_*(x,\HH^n|_E) =0$ for $\HH^n$-a.e.\ $x\in E$ implies that $E$ is purely $n$-unrectifiable. Recall that a set $F\subset\R^{n+1}$ is called $n$-rectifiable if there are countably many
Lipschitz manifolds $\Gamma_1,\Gamma_2,\ldots$ such that
$$\HH^n\Bigl(F\setminus \bigcup_i \Gamma_i\Bigr)=0.$$
On the other hand, $F$ is called purely $n$-unrectifiable if it intersects any $n$-rectifiable set at most in a set of zero $\HH^n$ measure.

A measure $\mu$ is called $n$-AD-regular if there exists $C \geq 1$ such that
$$C^{-1} r^n \leq \mu(E\cap B(x,r))\leq C r^n\quad \mbox{ for all $x\in \supp \mu$, $0<r<\diam(\supp \mu)$}.$$
In particular, a set $E\subset\R^{n+1}$ is $n$-AD-regular if $\mu= \HH^n|_E$ is $n$-AD regular.
For this type of sets, Nazarov, Tolsa and Volberg \cite{NToV} showed that
 the $L^2(\HH^n|_E)$ boundedness of $\RR_{\HH^n|_E}^n$ is equivalent to the uniform $n$-rectifiability of $E$ (see the next section for the precise definition of uniform $n$-rectifiability).
 This is the so called David-Semmes problem,
which is still open for the $k$-dimensional Riesz transform when $k$ is different from $1$ or $n$ in $\R^{n+1}$. By combining the solution of the David-Semmes problem for the $n$-AD-regular case in \cite{NToV} with the aforementioned result of Eiderman, Nazarov, and Volberg, it was shown in \cite{NToV-pubmat} that $\RR_{\HH^n|_E}^n$ is not bounded in $L^2(\HH^n|_E)$ whenever $E$ is purely $n$-unrectifiable.


The above results about the connection between Riesz transforms and rectifiability have been very fruitful
for the study of the geometric properties of harmonic measure in the recent works
\cite{ahm3tv}, \cite{MoTo}, \cite{amt-cpam}, \cite{amtv-ajm}, \cite{gmt-duke}, and \cite{amt-quantcpam}. It is natural then to try to extend these theorems to the case of elliptic measure associated with elliptic operators with H\"older continuous coefficients, which is one of our motivations for the present work.  

More precisely, 
let $A=(a_{ij})_{1\leq, i,j \leq n+1}$ be an $(n+1)\times (n+1)$ matrix whose entries $a_{ij}:\R^{n+1} \to \R$  are measurable functions in $L^\infty(\R^{n+1})$. Assume also that there exists $\Lambda>0$ such that
\begin{align}\label{eqelliptic1}
&\Lambda^{-1}|\xi|^2\leq \langle A(x) \xi,\xi\rangle,\quad \mbox{ for all $\xi \in\R^{n+1}$ and a.e. $x\in\R^{n+1}$,}\\
&\langle A(x) \xi,\eta \rangle  \leq\Lambda |\xi| |\eta|, \quad \mbox{ for all $\xi, \eta \in\R^{n+1}$ and a.e. $x\in\R^{n+1}$.} \label{eqelliptic2}
\end{align}
We consider elliptic equations of the form
\begin{equation}\label{eq:ellipticeq}
L_A u(x):= -\mathrm{div}\left(A(\cdot) \nabla u (\cdot) \right)(x)=0,
\end{equation}
which are understood in the distributional sense.
We say that a function $u \in W^{1,2}_{\loc}(\om)$ is a {\it solution} of \eqref{eq:ellipticeq} or {\it $L_A$-harmonic} in an open set $\om \subset \R^{n+1}$ if
$$
\int A \nabla u \cdot \nabla \vphi = 0, \quad \mbox{ for all $\vphi \in C_c^\infty(\om)$.}
$$
Denote by $\mathcal{E}_A(x,y)$ or just $\mathcal{E}(x,y)$ when the matrix $A$ is clear from the context the {\it fundamental solution} for $L_A$ in $\R^{n+1}$, so that $L_A \mathcal{E}_A(x,y) = \delta_y$ in the distributional sense, where $\delta_y$ is the Dirac mass at the point $y \in \R^{n+1}$. See in \cite{HK} for its construction. The integral $\int \mathcal{E}_A(x,y)\,d\mu(y)$ is usually known as the {\it single layer potential} of $\mu$.
Consider the operator $T$ whose kernel is
\begin{equation}\label{eq:Kdef}
K(x,y) = \nabla_1 \mathcal{E}_A(x,y)
\end{equation}
(the subscript $1$ means that we take the gradient with respect to the first variable), so that for a measure $\mu$ we have
\begin{equation}\label{eq:Tmudef}
T\mu(x) = \int K(x,y) \,d\mu(y)
\end{equation}
when $x$ is away from $\mathrm{supp}(\mu)$.   That is, $T\mu$ is the gradient of the single layer potential of $\mu$.

Given a function $f\in L^1_{loc}(\mu)$,
set also
\begin{equation}\label{eq:Tfdef}
T_\mu f(x) = T(f\,\mu)(x) = \int K(x,y) f(y)\,d\mu(y),
\end{equation}
and, for $\ve>0$, consider the $\ve$-truncated version
$$
T_{\ve}\mu(x) = \int_{|x-y|>\ve} K(x,y) \,d\mu(y)
$$
For $f\in L^1_{loc}(\mu)$,
we also write $T_\mu f(x) = T(f\mu)(x)$ and
$T_{\mu,\ve} f(x) = T_\ve (f\mu)(x)$. We say that the operator $T_\mu$ is bounded on $L^2(\mu)$ if the operators $T_{\mu,\ve}$ are bounded in $L^2(\mu)$ uniformly on $\ve>0$.

In the special case that $A$ is the identity matrix, $-L_A$ is just the Laplacian and $T$ is the $n$-dimensional Riesz transform $\RR_\mu^n$  up to a multiplicative constant depending only on the dimension $n$.
In fact, the operator $T_\mu$ plays the same role in connection with elliptic measure as $\RR_\mu^n$ regarding
harmonic measure.

We will also assume that the matrix $A$ is H\"older continuous, that is, there exists $\alpha>0$ and $C_h>0$ such that
\begin{equation}\label{eq:Holdercont}
|a_{ij}(x)-a_{ij}(y)| \leq C_h |x-y|^\alpha\quad \mbox{ for all $x,y \in \mathbb{R}^{n+1}$,}
\end{equation}
for $1 \leq i,j \leq n+1$. This assumption is essential in this paper because it ensures that the kernel $K(\cdot,
\cdot)$ is locally of Calder\'on-Zygmund type. However, we remark that, in general, $K(\cdot,
\cdot)$ is neither homogeneous nor antisymmetric.
Nevertheless, when $\mu$ is uniformly $n$-rectifiable, the operator $T_\mu$
is still bounded in $L^2(\mu)$, analogously to the Riesz transform $\RR_\mu^n$. See Theorem \ref{teo2} below for
more details. 

The main result of this paper is the following.

\begin{ltheorem} \label{ENV.theoremA}
Let $\mu$ be a non-zero measure in $\mathbb R^{n+1}$  such that 
$0<\Theta^{n,*}(x,\mu)<\infty$ and
$\Theta_*^n(x,\mu)=0$ 
for $\mu$-a.e.\ $x\in\R^{n+1}$.  Let $A$ be an elliptic matrix satisfying \eqref{eqelliptic1}, \eqref{eqelliptic2} and \eqref{eq:Holdercont}, and let $T_\mu$ be the associated operator given by \eqref{eq:Tfdef}.
Then $T_\mu$ does not map $L^2(\mu)$ into itself.
\end{ltheorem}

As mentioned above,  when $A$ is the identity matrix, it turns out that $T_\mu=c\,\RR_\mu^n$ for some $c\in\R\setminus\{0\}$ and the result above was proved previously
by Eiderman, Nazarov and Volberg in \cite{ENV}.
As in this work, our main idea to prove Theorem \ref{ENV.theoremA} is to apply a variational argument which requires a maximum principle.
However, instead of the quasiorthogonality techniques from \cite{ENV}, we will
use orthogonality via a martingale difference decomposition involving the cubes of the David-Mattila lattice. Therefore, the 
general organization of the proof of Theorem \ref{ENV.theoremA} is quite different from the one in \cite{ENV}.

Finally, we would like to inform the reader about a future work by Prat, Puliatti and Tolsa \cite{PPT} which deals with the 
case that $\mu$ is an $n$-AD-regular measure in $\R^{n+1}$. It will be proved there that if $T_\mu$ is bounded in $L^2(\mu)$ (with $T_\mu$ as in Theorem \ref{ENV.theoremA}), then $\mu$ is uniformly $n$-rectifiable. This extends the results of Nazarov, Tolsa and Volberg \cite{NToV} for the Riesz transform to gradients of single layer potentials associated with elliptic operators with real H\"older continuous coefficients. The result in \cite{PPT} combined with Theorem \ref{ENV.theoremA} will imply that if $E\subset\R^{n+1}$ is a set with $\HH^n(E)<\infty$ such that $T_{\HH^n|_E}$ is bounded in $L^2(\HH^n|_E)$, then $E$ is $n$-rectifiable. This was previously proved for the $n$-dimensional Riesz transform in \cite{NToV-pubmat}.


\subsection*{Notation}
In this paper  we will use the letters $c,C$ to denote
constants (quite often absolute constants, perhaps depending on $n$) which may change their values at different
occurrences. On the other hand, constants with subscripts, such as $C_1$, do not change their values
at different occurrences.

We will write $a\lesssim b$ if there is $C>0$ so that $a\leq Cb$ and $a\lesssim_{t} b$ if the constant $C$ depends on the parameter $t$. We write $a\approx b$ to mean $a\lesssim b\lesssim a$ and define $a\approx_{t}b$ similarly. 

We denote the open ball of radius $r$ centered at $x$ by $B(x,r)$. For a ball $B=B(x,r)$ and $a>0$ we write $r(B)$ for its radius and $a B=B(x,a r)$. 

\vv
\section{$L^p$ boundedness of $T_\mu$ for uniformly rectifiable measures and fundamental solutions}\label{sec2}


For any uniformly elliptic matrix $A$ with H\"older continuous coefficients, one can show that $K(\cdot,\cdot)$ is locally a Calder\'on-Zygmund kernel:
\begin{lemma}\label{lemcz}
Let $A$ be an elliptic matrix satisfying \eqref{eqelliptic1}, \eqref{eqelliptic2} and \eqref{eq:Holdercont}. If $K(\cdot,\cdot)$ is given by \eqref{eq:Kdef}, then locally, it is a Calder\'on-Zygmund kernel. That is, for any given $R>0$,
\begin{itemize}
\item[(a)] $|K(x,y)| \lesssim |x-y|^{-n}$ for all $x,y\in \R^{n+1}$ with $x \not=y$ and $|x-y|\leq R$.
\item[(b)] $|K(x,y)-K(x,y')| + |K(y,x) - K(y',x)| \lesssim |y-y'|^{\gamma} |x-y|^{-n-\gamma}$ for all $y,y'\in B(x,R)$ with $2|y-y'| \leq |x-y|$.
\item[(c)] $|K(x,y)| \lesssim |x-y|^{\frac{1-n}2}$  for all $x,y\in \R^{n+1}$ with $|x-y|\geq 1$,
\end{itemize}
All the implicit constants in (a), (b) and (c) depend on $\Lambda$ and  $C_h$, while the ones in (a) and (b) depend also on $R$.
\end{lemma}

\begin{proof}
The lemma follows from standard arguments. For (a) and (b) see e.g. \cite[p.5]{KS11} for details.
To show (c) we can assume that $r:=|x-y|>8$
because otherwise this follows from the estimate (a) with $R=10$, say.
Since $\EE_A(\cdot,y)$ is $L_A$-harmonic away from $y$,  by Caccioppoli's inequality we have
$$\avint_{B(x,r/4)}|\nabla_1 \EE_A(z,y)|^2\,d\LL^{n+1}(z)
\lesssim \frac1{r^2}\avint_{B(x,r/2)}|\EE_A(z,y)|^2\,d\LL^{n+1}(z).$$
Using that $\EE_A(z,y)\lesssim |y-z|^{1-n}\approx r^{1-n}$ in the integral above, we get
\begin{align*}
\int_{B(x,2)}|\nabla_1 \EE_A(z,y)|^2\,d\LL^{n+1}(z) &\lesssim r^{n+1}
\avint_{B(x,r/4)}|\nabla_1 \EE_A(z,y)|^2\,d\LL^{n+1}(z)\\
& \lesssim \frac{r^{n+1}\,r^{2(1-n)}}{r^2} = r^{1-n}.
\end{align*}
Now, by standard results from elliptic PDE's,
$$\|\nabla_1 \EE_A(\cdot,y)\|_{C^\alpha(B(x,1))}\lesssim \|\nabla_1 \EE_A(\cdot,y)\|_{L^2(B(x,2))}\lesssim
r^{\frac{1-n}2} = |x-y|^{\frac{1-n}2},$$
which implies (c).
\end{proof}

\subsection{Reduction to constant, symmetric coefficients}
\label{ENV.subsecConst}


If $E$ is an elliptic matrix with real and constant coefficients, we denote
\begin{equation}\label{eq:Thetadef}
\Theta(x,y;E) : = \mathcal{E}_E(x,y),
\end{equation}
where $\mathcal{E}_E$ is the fundamental solution of the elliptic operator whose (constant) matrix is $E$. Note that $\mathcal{E}_E$ is symmetric, and moreover,
$$
\Theta(x,y;E)  = \Theta(x-y,0;E) = \Theta(y-x,0;E) = \Theta(y,x;E).   
$$
 Also, having fixed the constant matrix $E = (e_{ij})$, we may assume without loss of generality that $L_E$ has  symmetric coefficients. Indeed, if we denote by $E_{\mathrm{sym}}$ the matrix with entries $e_{ij}^{\mathrm{sym}} :=(e_{ij}+e_{ji})/2$, then  
\begin{align*}
-L_Eu& = \sum_{i,j} e_{ij} \partial_i  \partial_j u= \frac12\sum_{i,j} e_{ij} \partial_i  \partial_j u + \frac12\sum_{i,j} e_{ij} \partial_j \partial_i u\\ &=  \sum_{i,j} \frac{e_{ij} + e_{ji}}{2} \,\partial_i \partial_j u = -L_{\mathrm{sym}}  u.
\end{align*}
Therefore, it is clear that a function $u$ solves $L_Eu=-\mathrm{div} ( E \nabla u)=0$  if and only if it solves $L_{\mathrm{sym}} u =-\mathrm{div} (E_{\mathrm{sym}} \nabla u) =0$.

It turns out that in small scales, we may approximate  our non-symmetric kernel $K_A(\cdot, \cdot)$ (associated with the matrix $A$) by another one $\Theta(\cdot,\cdot; E_A)$ (associated with a constant coefficient matrix $E_A$). The precise result is the following:

\begin{lemma} \label{ENV.constantCoef}
Let $A$ be an elliptic matrix satisfying \eqref{eqelliptic1}, \eqref{eqelliptic2} and \eqref{eq:Holdercont}. Let also $\Theta(\cdot,\cdot; \cdot)$ be given by \eqref{eq:Thetadef}. Then for $R>0$ and for all $x,y\in B(0,R)$ we have
\begin{enumerate}
\item $|\mathcal{E}_A (x,y) - \Theta(x,y;A(x))| \lesssim |x-y|^{\alpha-n+1}$.
\item $|\nabla_1\mathcal{E}_A (x,y) - \nabla_1\Theta(x,y;A(x))| \lesssim |x-y|^{\alpha-n}$.
\item $|\nabla_1\mathcal{E}_A (x,y) - \nabla_1\Theta(x,y;A(y))| \lesssim |x-y|^{\alpha-n}$.
\end{enumerate}
Similar inequalities hold if we reverse the roles of $x$ and $y$ and replace $\nabla_1$ by $\nabla_2$.
All the implicit constants depend on $\Lambda$,  $C_h$, and  $R$.
\end{lemma}

For the proof of the above result, see \cite[Lemma 2.2]{KS11}. This entails as a very easy consequence that for any ball $B$ and $x\in B$,
\begin{equation}\label{eqap*1}
\int_B |\nabla_1\mathcal{E}_A (x,y) - \nabla_1\Theta(x,y;A(x))| d\mu(x) \lesssim r(B)^\alpha,
\end{equation}
assuming that $\mu$ has $n$-polynomial growth, i.e., there exists $c_0>0$ such that
$$\mu(B(x,r))\leq c_0\,r^n\quad\mbox{ for all $x\in\R^{n+1}$ and all $r>0$.}$$

\subsection{Uniform rectifiability and singular integrals}

We say that a measure $\mu$ is uniformly $n$-rectifiable in $\R^d$, for $1 \leq n\leq d$, if it is 
$n$-AD-regular and
there exist $\theta, M >0$ such that for all $x \in \supp(\mu)$ and all $r>0$ 
there is a Lipschitz mapping $g$ from the ball $B_n(0,r)$ in $\R^{n}$ to $\R^d$ with $\text{Lip}(g) \leq M$ such that$$
\mu (B(x,r)\cap g(B_{n}(0,r)))\geq \theta r^{n}.$$
The notion of uniform $n$-rectifiability was introduced by David and Semmes and it is a quantitative version of 
$n$-rectifiability.

As we saw in the introduction, a set $E\subset\R^d$ is called $n$-AD-regular if $\HH^n|_E$ is $n$-AD-regular, and it is called
uniformly $n$-rectifiable if $\HH^n|_E$ is uniformly  $n$-rectifiable. See \cite{DS2} for more information on uniform rectifiability. The following theorem is due to G. David and S. Semmes:

\begin{theorem}\label{thm:TDavid}
Let $\mu$ be a uniformly $n$-rectifiable measure in $\R^d$. Let
$K: \R^{d} \setminus\{0\} \to \R$ be a kernel which is odd and homogeneous of degree $-n$, i.e., $K(-x)=-K(x)$ and $K(\lambda x) = \lambda^{-n} K(x)$, such that for some  $M \in \mathbb N$ it holds
\begin{align}\label{eq:djkernel}
|\nabla_j K (x)| \lesssim_n C(j)\, |x|^{-n-j}, \quad \textup{for all} \,\, 0 \leq j \leq M \,\,  \textup{and}\,\,  x \in \R^{n+1} \setminus \{0\}.
\end{align}
Consider the associated operator
$$T_{K, \mu} f(x) = \int K(x-y) f(y) d \mu(y).$$
 Then, for $1<p<\infty$, $T_{K,\mu}:L^p(\mu) \to L^p(\mu)$ is bounded, with bounds that depend on $p$, on the constants $C(j)$, and on the uniform rectifiability constants of $\mu$.
\end{theorem}

One can use Theorem \ref{thm:TDavid} to prove the following proposition:

\begin{proposition}[Proposition 1.2 in \cite{MiTa}]\label{prop:Mit-Tay}
Let $\mu$ be  a uniformly $n$-rectifiable measure in $\R^{n+1}$. There exists $M=M(n)$ such that the following holds. Let $b(x,y)$ be odd in $x$ and homogeneous of degree $-n$ in $x$, and assume that $\partial^\beta_x b(x,y)$ is continuous and bounded in $\mathbb S^n \times \R^{n+1}$, for $|\beta| \leq M$. Then $b(x-y, x)$ is the kernel of an operator $\bf B_\mu$, bounded on $L^p(\mu)$, for $1 <p <\infty$.
\end{proposition}

The above proposition was proved under the assumption that $\mu= \HH^n_{|\Gamma}$, where $\Gamma$ is a Lipschitz graph. However, the same proof works for uniformly $n$-rectifiable measures using the expansion of the kernel in spherical harmonics once we assume Theorem \ref{thm:TDavid}. 

 Next we apply the preceding proposition to the operator $T_\mu$ associated with the elliptic matrix $A$.

\begin{theorem}\label{teo2}
Let $\mu$ be a uniformly $n$-rectifiable measure with compact support in $\R^{n+1}$. Let $A$ be an elliptic matrix satisfying \eqref{eqelliptic1}, \eqref{eqelliptic2} and \eqref{eq:Holdercont},
and let $T_\mu$ be the associated operator given by \eqref{eq:Tfdef}.
 Then, for $1<p<\infty$, $T_\mu: L^p(\mu) \to  L^p(\mu)$ is bounded.
\end{theorem}

\begin{proof}
Let $\wt K(z,w):= \nabla_1\Theta(z,0;A(w))$, i.e., the gradient of the fundamental solution associated with the constant coefficients matrix $A(w)$ with pole at the origin. Then, from the properties of $\Theta(\cdot,0;A(w))$, we have that $\wt K(z,w)$ is odd in $z$, homogeneous of degree $-n$ in $z$ and satisfies \eqref{eq:djkernel} for each fixed $w$.  We may write
$$ \nabla_1\Theta(x,y;A(x)) = \nabla_1\Theta(x-y,0;A(x))=\wt K(x-y,x)$$
and similarly,
$$ \nabla_2\Theta(x,y;A(y)) = \nabla_1\Theta(y-x,0;A(y))=\wt K(y-x,y).$$
In light of Proposition \ref{prop:Mit-Tay}, it follows that the operator
$$\wt T_\mu f(x) := \int \wt K(x-y,x)\,f(y)\,d\mu(y)$$
is bounded in $L^p(\mu)$.
Thus, to prove the theorem it suffices to show that the operator 
$$ Q_\mu f(x)=\int [ \nabla_1 \EE_{A}(x, y) - \nabla_1 \Theta(x,y; A(x)) ] f(y) \,d\mu(y)=: \int q(x,y) f(y) \,d\mu(y)$$
is bounded in $L^p(\mu)$.  But this immediately follows from Lemma \ref{ENV.constantCoef}, since the kernel of $Q_\mu$ satisfies the bound
$$|q(x,y)| \lesssim_\alpha |x-y|^{\alpha -n},$$
which, in turn, implies that $Q_\mu$ is compact in $L^p(\mu)$ and thus bounded.
\end{proof}

\vv


\section{The dyadic lattice of David and Mattila}\label{sec:DMlatt}

In this section, for a given measure $\mu$ in $\R^d$, we construct the David-Mattila cubes associated to $\mu$ (see \cite{david-mattila}), that is, we consider a sequence $\DD=\{\DD_{k}\}_{k \geq k_0}$ of nested partitions of $\supp(\mu)$ |whose elements we shall call cubes| with some remarkable properties which we summarize in the next lemma.

\begin{lemma}[Theorem 3.2 in \cite{david-mattila}]
\label{lemcubs}
Let $\mu$ be a compactly supported Radon measure in $\R^{d}$.
Consider two constants $C_0>1$ and $A_0>5000\,C_0$. 
Then there exists a sequence $\DD= \cup_{k \geq k_0} \DD_k$ of partitions of $\supp\mu$ into
Borel subsets $Q$ with the following properties:
\begin{itemize}
\item For each integer $k\geq k_0$, $\supp\mu$ is the disjoint union of the ``cubes'' $Q$, $Q\in\DD_{k}$, and
if $k<l$, $Q\in\DD_{k}$, and $R\in\DD_{l}$, then either $Q\cap R=\varnothing$ or else $Q\subset R$.
\vspace{2mm}

\item The general position of the cubes $Q$ can be described as follows. For each $k\geq k_0$ and each cube $Q\in\DD_{k}$, there is a ball $B(Q)=B(x_Q,r(Q))$ such that
$$x_Q\in \supp\mu, \qquad A_0^{-k}\leq r(Q)\leq C_0\,A_0^{-k},$$
$$\supp\mu\cap B(Q)\subset Q\subset \supp\mu\cap 28B(Q)=\supp\mu \cap B(x_Q,28r(Q)),$$
and
$$\mbox{the balls\, $5B(Q)$, $Q\in\DD_{k}$, are disjoint.}$$
The balls $\frac12B(Q)$ and $\frac12B(Q')$ associated with different cubes $Q$ and $Q'$ are disjoint unless $Q \subset Q'$ or $Q' \subset Q$.
\vspace{2mm}
\item The cubes $Q\in\DD_{k}$ have small boundaries. That is, for each $Q\in\DD_{k}$ and each
integer $l\geq0$, set
$$N_l^{ext}(Q)= \{x\in \supp\mu\setminus Q:\,\dist(x,Q)< A_0^{-k-l}\},$$
$$N_l^{int}(Q)= \{x\in Q:\,\dist(x,\supp\mu\setminus Q)< A_0^{-k-l}\},$$
and
$$N_l(Q)= N_l^{ext}(Q) \cup N_l^{int}(Q).$$
Then
\begin{equation}\label{eqsmb2}
\mu(N_l(Q))\leq (C^{-1}C_0^{-3d-1}A_0)^{-l}\,\mu(90B(Q)).
\end{equation}
\vspace{2mm}

\item Denote by $\DD_{k}^{db}$ the family of cubes $Q\in\DD_{k}$ for which
\begin{equation}\label{eqdob22}
\mu(100B(Q))\leq C_0\,\mu(B(Q)).
\end{equation}
We have that $r(Q)=A_0^{-k}$ when $Q\in\DD_{k}\setminus \DD_{k}^{db}$
and
\begin{equation}\label{eqdob23}
\mu(100B(Q))\leq C_0^{-l}\,\mu(100^{l+1}B(Q))\quad
\mbox{for all $l\geq1$ with $100^l\leq C_0$ and $Q\in\DD_{k}\setminus \DD_{k}^{db}$.}
\end{equation}
\end{itemize}
\end{lemma}

\vspace{2mm}

Observe that the families $\DD_{k}$ are only defined for $k\geq k_0$, and so the diameters of the cubes from $\DD$ are uniformly
bounded from above.  Further, we assume that $Q_{k_0}\equiv\supp\mu$.
Given $Q\in\DD_{k}$, we denote $J(Q)=k$, and we set
$\ell(Q)= 56\,C_0\,A_0^{-k}$ and we call it the side length of $Q$. Notice that 
$$C_0^{-1}\ell(Q)\leq \diam(28B(Q))\leq\ell(Q).$$
Observe that $r(Q)\approx\diam(Q)\approx\ell(Q)$.
Also we call $x_Q$ the center of $Q$, and the cube $Q'\in \DD_{k-1}$ such that $Q'\supset Q$ the parent of $Q$.
 We set
$B_Q=28 B(Q)=B(x_Q,28\,r(Q))$, so that 
$$\supp\mu\cap \tfrac1{28}B_Q\subset Q\subset B_Q.$$

We assume $A_0$ big enough so that the constant $C^{-1}C_0^{-3d-1}A_0$ in 
\eqref{eqsmb2} satisfies 
$$C^{-1}C_0^{-3d-1}A_0>A_0^{1/2}>10.$$
Then we deduce that, for all $0<\lambda\leq1$,
\begin{align}\label{eqfk490}\nonumber
\mu\bigl(\{x\in Q:\dist(x,\supp\mu\setminus Q)\leq \lambda\,\ell(Q)\}\bigr) + 
\mu\bigl(\bigl\{x\in 3.5B_Q\setminus Q:&\dist(x,Q)\leq \lambda\,\ell(Q)\}\bigr)\\
&\quad\,\leq
c\,\lambda^{1/2}\,\mu(3.5B_Q).
\end{align}

We denote
$\DD^{db}=\bigcup_{k\geq k_0}\DD_{k}^{db}$.
Note that, in particular, from \eqref{eqdob22} it follows that
\begin{equation}\label{eqdob*}
\mu(3.5B_{Q})\leq \mu(100B(Q))\leq C_0\,\mu(Q)\qquad\mbox{if $Q\in\DD^{db}.$}
\end{equation}
For this reason we call the cubes from $\DD^{db}$ doubling. 
Given $Q\in\DD$, we denote by $\DD(Q)$
the family of cubes from $\DD$ which are contained in $Q$. Analogously,
we write $\DD^{db}(Q) = \DD^{db}\cap\DD(Q)$. We will also use the following properties of the construction: 

\begin{lemma}[Lemma 5.28 in \cite{david-mattila}]\label{lemcobdob}
Let $Q\in\DD$. Suppose that the constants $A_0$ and $C_0$ in Lemma \ref{lemcubs} are
chosen suitably. Then there exists a family of
doubling cubes $\{Q_i\}_{i\in I}\subset \DD^{db}$, with
$Q_i\subset Q$ for all $i$, such that their union covers $\mu$-almost all $Q$.
\end{lemma}

\begin{lemma}[Lemma 5.31 in \cite{david-mattila}]\label{lemcad22}
Let $Q\in\DD$ and let $R\subset Q$ be a cube such that all the intermediate cubes $S$,
$R\subsetneq S\subsetneq Q$ are non-doubling (i.e.\ belong to $\DD\setminus \DD^{db}$).
Then%
\begin{equation}\label{eqdk88}
\mu(100B(R))\leq A_0^{-10d(J(R)-J(Q)-1)}\mu(100B(Q)).
\end{equation}
\end{lemma}

Given a ball  $B\subset \R^{d}$ and a fixed $1\leq n\leq d$, we consider its $n$-dimensional density:
$$\Theta_\mu(B)= \frac{\mu(B)}{\diam(B)^n}.$$
For a cube $Q \in \DD$, we also set
$$
\Theta_\mu(Q)=\frac{\mu(Q)}{\ell(Q)^n}.
$$

From the preceding lemma we deduce:

\vspace{2mm}
\begin{lemma}\label{lemcad23}
Let $1\leq n\leq d$ and 
let $Q,R\in\DD$ be as in Lemma \ref{lemcad22}.
Then
$$\Theta_\mu(100B(R))\leq (C_0A_0)^d\,A_0^{-9d(J(R)-J(Q)-1)}\,\Theta_\mu(100B(Q)).$$
\end{lemma}


\section{Initial reductions via a martingale difference decomposition}

We assume that $\mu$ is a compactly supported non-zero measure in $\mathbb R^{n+1}$ 
such that 
$$
\Theta^{n,*}(x,\mu)>0
\quad
\textup{and}
\quad
\Theta_*^n(x,\mu)=0
$$
for $\mu$-a.e.\ $x\in\R^{n+1}$.
 To show that $T_\mu$ is not bounded in $L^2(\mu)$, 
 by replacing $\mu$ by its restriction to a suitable subset with positive $\mu$-measure, 
we may assume that there
exists some constant $\tau_0>0$ such that $\Theta^{n,*}(x,\mu)>\tau_0$ for $\mu$-a.e.\ $x$ and also that $\mu$ has
$n$-polynomial growth with constant $c_0$. 

From now on we also assume that the last two conditions hold. In what follows we allow constants, explicit or implicit in the relations $\approx$ and
$\lesssim$, to depend on the parameters of the David-Mattila lattice $C_0,A_0$ and on the polynomial growth constant $c_0$. We first need a technical result:

\begin{lemma}\label{lemdens}
The following hold:
\begin{itemize}
\item[(a)] For $\mu$-a.e.\ $x\in\R^{n+1}$ there exists a sequence of cubes $Q_k\in\DD^{db}$ such that $x\in Q_k$, $\ell(Q_k)\to0$,
and $\Theta_\mu(Q_k)>c\,\tau_0$, with $c$ depending on $n$ and the parameters of the David-Mattila lattice.

\item[(b)] Let $A>1$ and $0<\delta<1$.
For $\mu$-a.e.\ $x\in\R^{n+1}$ there exists a sequence of cubes $Q_k\in\DD^{db}$ such that $x\in Q_k$, $\ell(Q_k)\to0$,
and $\Theta_\mu(AB_{Q_k})\leq \delta$.
\end{itemize}
\end{lemma}
 
\begin{proof}
To prove (a), let $x\in\R^{n+1}$ be such that $\Theta^{n,*}(x,\mu)\geq\tau_0$. 
We have to show that for any given $\ell_0>0$ there exists a cube $Q\in\DD^{db}$ such that
$x\in Q$, $\ell(Q)\leq \ell_0$,
and $\Theta_\mu(Q)\gtrsim\tau_0$.
To this end, let $B(x,r)$ be a ball such that $0<r\leq c_1\,\ell_0$ and $\Theta_\mu(B(x,r))\geq \tau_0/2$,
with $c_1<1$ to be fixed in a moment. Let $R_0\in\DD$ be the smallest cube such that $B(x,r)\subset 2B_{R_0}$. Since
$\ell(R_0)\approx r$, we have 
$$\Theta_\mu(2B_{R_0})\gtrsim \Theta_\mu(B(x,r))\geq \frac12\,\tau_0.$$
Let $Q=Q(R_0)\in\DD^{db}$ be the smallest doubling cube that contains $R_0$ (such a cube exists because $Q_{k_0}\equiv\supp\mu\in
\DD^{db}$).
For $j\geq 0$, denote by $R_j$ the $j$-th ancestor of $R_0$ (i.e.\ $R_j\in\DD$ is such that $R_0\subset R_j$ and $\ell(R_j)=A_0^j\,\ell(R_0)$).
Let $i\geq 0$ be such that $Q(R_0)= R_i$.  If $i \leq 10$, then it is clear that  $\Theta_\mu(Q) = \Theta_\mu({R_i})\gtrsim \tau_0$. Assume now that $i>10$. Then since the cubes $R_1,\ldots,R_{i-1}$ do not belong to $\DD^{db}$, by Lemma \ref{lemcad23} we have
\begin{align*}
\tau_0\lesssim\Theta_\mu(2B_{R_0})\lesssim\Theta_\mu(100B(R_0))&\leq (C_0A_0)^{n+1}\,A_0^{-9(n+1)(i-1)}\,\Theta_\mu(100B(R_i))\\
&\lesssim C_0^{2(n+1)} A_0^{-8(n+1)i}\,\Theta_\mu({R_i}),
\end{align*}
where in the last inequality we used that $R_i \in \DD^{db}$.
This implies that $\Theta_\mu(Q) = \Theta_\mu({R_i})\gtrsim \tau_0$. Also, taking into account that 
$\Theta_\mu({R_i})\lesssim c_0$, we infer that 
$\tau_0\lesssim c_0\,C_0^{2(n+1)} A_0^{-8(n+1)i},$
and thus $i$ is bounded above by some constant depending on $\tau_0$, $c_0$, $A_0$ and $C_0$, which implies that
$$\ell(Q)\leq C(\tau_0,c_0, A_0, C_0)\,\ell(R_0)\leq C'(\tau_0,c_0, A_0, C_0)\,c_1\,\ell_0.$$
Hence, choosing $c_1= C'(\tau_0,c_0, A_0, C_0)^{-1}$ we are done.
\vspace{2mm}

Let us turn our attention to (b).
Let $x\in\R^{n+1}$ be such that $\Theta^{n}_*(x,\mu)=0$ 
and such that there exists a sequence of doubling cubes $P_k\in\DD^{db}$ such that $x\in P_k$, $\ell(P_k)\to0$.
 By Lemma \ref{lemcobdob}, the set of such points $x\in\R^{n+1}$ has full measure.
Let 
 $\ell_0>0$. We wish to find a cube $Q\in\DD^{db}$ such that
$x\in Q$, $\ell(Q)\leq \ell_0$,
and $\Theta_\mu(AB_Q)\leq \delta$.
To this end, let $B(x,r)$ be a ball such that $0<r\leq \ell_0$ and $\Theta_\mu(B(x,Ar))\leq \delta A^{-n}$.
Let $S_0\in\DD$ be the largest cube such that $x\in S_0$, $100B(S_0)\subset B(x,r)$,  and $r(B(S_0)) \approx r$.
For $j\geq 0$, denote by $S_j$ the $j$-th descendant of $S_0$ that contains $x$
(i.e.\ $S_j\in\DD$ is such that $x\in S_j\subset S_0$ and $\ell(S_j)=A_0^{-j}\,\ell(S_0)$).
Let $i\geq 0$ be the least integer such that $Q:= S_i$ is doubling.
Since the cubes $S_1,\ldots,S_{i-1}$ do not belong to $\DD^{db}$, by Lemma \ref{lemcad23} we have
$$\Theta_\mu(100B(S_j))\leq  (C_0A_0)^{n+1}\,A_0^{-9(n+1)(j-1)}\,\Theta_\mu(100B(S_0))\lesssim_{A_0, C_0} \Theta_\mu(B(x,r))
$$
for $j=1,\ldots,i$. This implies that any ball $B$ concentric with $B_{S_i}$ such that $100B(S_i))\subset B \subset B(x,r)$ satisfies
$$\Theta_\mu(B)\lesssim \Theta_\mu(B(x,r))\leq A^n\,\Theta_\mu(B(x,Ar))\leq \delta.$$
The same estimate holds if $B\subset B(x,Ar)$ and $r(B)\geq r$.
So we infer that $\Theta_\mu(AB_{S_i})\lesssim \delta$, and by adjusting suitably the initial choice of $\delta>0$ if necessary, (b) follows.
\end{proof}

\vspace{2mm}

 In view of the preceding lemma, we fix $A$ big enough and $\delta$ small enough (to be chosen below) and for each $Q \in \DD$ we define
$$
\mathrm{HD}(Q)=\left\{R \subsetneq Q, \; R \in \DD^{db}, \Theta_\mu(R) > \tau, \; R \; \mathrm{maximal}\right\},
$$
where $\tau=c\,\tau_0$, with $c$ as in (a) of Lemma \eqref{lemdens}. We also set
$$
\mathrm{LD}(Q)=\left\{R \subsetneq Q, \; R \in \DD^{db}, \Theta_\mu(AB_R) \leq \delta, \; R \; \mathrm{maximal}\right\}.
$$
We construct now a subfamily of $\DD$ that we will denote by $\Sigma$. First, we pick $\Sigma_0 = \{Q_{k_0}\}$. Notice that we can assume that $\Theta_\mu(Q_{k_0}) \leq \delta$ (by possibly enlarging the cube $Q_{k_0}$ which already contains the support of $\mu$). Then, for a general cube $Q$ we define
$$
\Sigma_1(Q) = \bigcup_{R \in \mathrm{HD}(Q)} \mathrm{LD}(R).
$$
By the above discussion, $\Sigma_1(Q)$ is always a partition of $Q$ up to a set of $\mu$-measure $0$. We now inductively define
$$
\Sigma_{k+1} = \bigcup_{Q \in \Sigma_k} \Sigma_1(Q), \;\; k\geq 0.
$$
Of course, $\Sigma=\{\Sigma_k\}_{k \geq 0}$ is a filtration of $\mathrm{supp}(\mu)$ modulo a set of $\mu$-measure $0$
composed of cubes of low density. So if we denote 
$$\Delta_Q f = \sum_{S \in \Sigma_1(Q)} \langle f\rangle_S \chi_S - \langle f \rangle_Q \chi_Q,$$
  we may write, in the $L^2(\mu)$ sense,
\begin{align*}
f =  \langle f\rangle_{Q_{k_0}} +\sum_{Q\in \Sigma} \Delta_Q f
 = \langle f\rangle_{Q_{k_0}} + \sum_{Q \in \Sigma} \sum_{S \in \Sigma_1(Q)} \left(\langle f \rangle_S - \langle f \rangle_Q \right) \chi_S.
\end{align*}
We apply this decomposition to $T\mu$ to obtain, by orthogonality of the martingale differences,
$$
\|T\mu\|_{L^2(\mu)}^2 =  (\langle T\mu\rangle_{Q_{k_0}})^2\,\mu(Q_{k_0})+ \sum_{Q \in \Sigma} \|\Delta_Q (T\mu)\|_{L^2(\mu)}^2.
$$
The following proposition is the key step for the proof of Theorem \ref{ENV.theoremA}:

\begin{proposition}\label{ENV.mainlemma}
Suppose that $T_\mu$ is bounded in $L^2(\mu)$. There exists $N_0$ such that if $Q \in \Sigma_N$, $N> N_0$, and $\delta$ is chosen small enough, then
$$
\|\Delta_Q(T\mu)\|_{L^{2}(\mu)}^2 \gtrsim_\tau \mu(Q).
$$
\end{proposition}

Of course, Proposition \ref{ENV.mainlemma} immediately implies Theorem \ref{ENV.theoremA} by contradiction and via the martingale decomposition above. The rest of this paper is therefore devoted to the proof of Proposition \ref{ENV.mainlemma}.


\section{Localization and approximation of $\mu$} 

From now on we assume that $T_\mu$ is bounded in $L^2(\mu)$. To prove Proposition \ref{ENV.mainlemma} we have to estimate $\|\Delta_Q(T\mu)\|_{L^{2}(\mu)}$ from below for each fixed $Q\in\DD$. We assume that $\diam(Q)\leq 1/2$, so that the estimates (a) and (b) in Lemma
\ref{lemcz} hold for all $x,y\in Q$, and even for $x,y$ in a small neighborhood of $Q$ (with a fixed $R=1$, say).

\vv
The first step in the argument is a change of measure in order to work with another that is absolutely continuous with respect to the Lebesgue measure on $\mathbb{R}^{n+1}$ and that is supported close to $Q$.
This will make the application of a suitable maximum principle possible. The family $\Sigma_1(Q)$ may consist of an infinite number of cubes.
For technical reasons, it is convenient to consider a finite subfamily of $\Sigma_1(Q)$ which contains a very big proportion of $\mu(Q)$. So, for a small $\ve_0>0$ to be chosen below, we let $\Sigma_1'(Q)$ be a {\em finite} subfamily of $\Sigma_1(Q)$
such that
\begin{equation}\label{stop00}
\mu\biggl(\,\bigcup_{S\in \Sigma_1'(Q)} S\biggr)> (1-\ve_0)\,\mu(Q).
\end{equation}

\vv
We will now define some  auxiliary regions $I_{\kappa_0}(S)$ associated with each $S\in\Sigma_1'(Q)$.
Given a small constant $0<\kappa_0\ll1$ (to be fixed below) and $S\in\Sigma_1'(Q)$, we denote
\begin{equation}\label{eqik00}
I_{\kappa_0}(S) = \{x\in S:\dist(x,\supp\mu\setminus S)\geq \kappa_0\ell(S)\}.
\end{equation}
So $I_{\kappa_0}(S)$ is some kind of inner subset of $S$. 
Observe that, by the doubling property and the small boundary condition \eqref{eqfk490} of $S\in\Sigma_1'(Q)$, we have
$$\mu(S\setminus I_{\kappa_0}(S))\lesssim \kappa_0^{1/2}\,\mu(3.5B_S) \lesssim \kappa_0^{1/2}\,\mu(S).$$
Next we consider the function
$$\tilde{\chi}_Q = \sum_{S\in\Sigma_1'(Q)}\chi_{I_{\kappa_0}(S)},$$
so that we have
$$\|\tilde{\chi}_Q-\chi_Q\|_{L^2(\mu)}^2 = \sum_{S\in\Sigma_1(Q)\setminus \Sigma_1'(Q)} \mu(S) + 
\sum_{S\in \Sigma_1'(Q)} \mu(S\setminus I_{\kappa_0}(S)) \leq \ve_0\,\mu(Q) + c\,\kappa_0^{1/2}\,\mu(Q).$$
We also denote by $\eta$ the auxiliary measure
$$\eta=\tilde{\chi}_Q\,\mu.$$

We consider the approximate measure
$$
\sigma = \sigma_Q = \sum_{S \in \Sigma_1'(Q)} \frac{\mu(I_{\kappa_0}(S))}{\mathcal{L}^{n+1}(\tfrac14B(S))} \mathcal{L}^{n+1}|_{\tfrac14B(S)}
$$
Notice that the new measure $\sigma$ satisfies 
\begin{equation}\label{eq:sigma=muQ}
\sigma(\R^{n+1}) = \int \tilde{\chi}_Q\,d\mu =\eta(\R^{n+1})\approx \mu(Q),
\end{equation}
since  we assume
$\ve_0\ll1$ and $\kappa_0\ll1$.
The objective of this section is to prove the following result.

\begin{lemma} \label{ENV.approximation}
For any $\ve>0$, if $Q\in\Sigma_N$ and $N$ is big enough, 
$\kappa_0$ and $\ve_0$ are small enough, $A$ is big enough, and $\delta$ is small enough (depending also on $A$ and $\kappa_0$), then
$$
\|\Delta_Q T \mu\|_{L^2(\mu)}^2 \geq c \|T \sigma\|_{L^2(\sigma)}^2 - \ve\,\mu(Q),
$$
where $c>0$ is an absolute constant.
\end{lemma}

Proposition \ref{ENV.mainlemma} immediately follows from Lemma \ref{ENV.approximation} and  \eqref{eq:sigma=muQ} once we show that
\begin{equation}\label{eq:Tsigma-lowerbnd}
\|T\sigma\|_{L^2(\sigma)}^2 \gtrsim_\tau \sigma(\R^{n+1})
\end{equation}
for $Q \in \Sigma_k$, $k > N_0$. 

\vv

The proof of Lemma \ref{ENV.approximation} will be carried out in several steps. The first one is the following:

\begin{lemma}\label{lemstep1}
For any $\ve>0$, if $Q\in\Sigma_N$ and $N$ is big enough, 
$A$ is chosen big enough, and $\delta$ small enough (depending also on $A$), then
we have
$$\Bigl\|\Delta_Q T \mu - \sum_{S\in\Sigma_1(Q)} \langle T_\mu\chi_Q\rangle_S\,\chi_S
\Bigr\|_{L^2(\mu)}^2\lesssim \varepsilon\,\mu(Q).$$
\end{lemma}

\begin{proof}
Let $0<\lambda<1$ be a small parameter to be chosen below. We denote
$$Q^\lambda = \{y\in \supp\mu:\dist(y,Q)\leq\lambda\,\ell(Q)\}.$$
Note that, by \eqref{eqfk490},
\begin{equation}\label{eqthin1}
\mu(Q^\lambda\setminus Q)\lesssim \lambda^{1/2}\,\mu(3.5B_Q)\lesssim \lambda^{1/2}\,\mu(Q),
\end{equation}
using also that $Q$ is doubling.

For $x\in S\in\Sigma_1(Q)$, we split
\begin{align}\label{eqsplit22}
\Delta_Q T\mu(x) & = \langle T\mu\rangle_S - \langle T\mu\rangle_Q\\
& =\langle T_\mu\chi_{\R^{n+1}\setminus AB_Q} \rangle_S + \langle T_\mu\chi_{AB_Q\setminus Q^\lambda} \rangle_S + \langle T_\mu\chi_{Q^\lambda\setminus Q} \rangle_S + \langle T_\mu\chi_{Q} \rangle_S
\nonumber\\
&\quad 
- \langle T_\mu\chi_{\R^{n+1}\setminus AB_Q} \rangle_Q
- \langle T_\mu\chi_{AB_Q\setminus Q^\lambda} \rangle_Q - \langle T_\mu\chi_{Q^\lambda\setminus Q} \rangle_Q. 
- \langle T_\mu\chi_{Q} \rangle_Q.\nonumber
\end{align}

Observe that
$$|\langle T_\mu\chi_{\R^{n+1}\setminus AB_Q} \rangle_S - \langle T_\mu\chi_{\R^{n+1}\setminus AB_Q} \rangle_Q|\leq \sup_{y,y'\in Q} |T_\mu\chi_{\R^{n+1}\setminus AB_Q}(y) - T_\mu\chi_{\R^{n+1}\setminus AB_Q}
(y')|.$$
To estimate the supremum on the right hand side, note that for $y,y'\in Q$,
\begin{align*}
|T_\mu\chi_{\R^{n+1}\setminus AB_Q}(y) - T_\mu\chi_{\R^{n+1}\setminus AB_Q}
(y')| & \leq \int_{z\in \R^{n+1}\setminus AB_Q} |K(y,z)-K(y',z)|\,d\mu(z) \\
& \lesssim \int_{z\in \R^{n+1}\setminus AB_Q} \frac{\ell(Q)^\gamma}{|x_Q-z|^{n+\gamma}}\,d\mu(z).
\end{align*}
By standard estimates, using the polynomial growth of $\mu$, it follows easily that 
$$\int_{z\in \R^{n+1}\setminus AB_Q} \frac{\ell(Q)^\gamma}{|x_Q-z|^{n+\gamma}}\,d\mu(z)\lesssim A^{-\gamma}.$$
 Hence,
$$|\langle T_\mu\chi_{\R^{n+1}\setminus AB_Q} \rangle_S - \langle T_\mu\chi_{\R^{n+1}\setminus AB_Q} \rangle_Q|
\lesssim A^{-\gamma}.$$

Next we estimate the terms $\langle T_\mu\chi_{AB_Q\setminus Q^\lambda} \rangle_S$ and $\langle T_\mu\chi_{AB_Q\setminus Q^\lambda} \rangle_Q$ on the right hand side of \eqref{eqsplit22}.
To this end, we write 
\begin{align*}
|\langle T_\mu\chi_{AB_Q\setminus Q^\lambda} \rangle_S| + |\langle T_\mu\chi_{AB_Q\setminus Q^\lambda} \rangle_Q|&\leq 2 \,\sup_{y\in Q} | T_\mu\chi_{AB_Q\setminus Q^\lambda}(y)| \\
&\leq 2
\sup_{y\in Q} \int_{AB_Q\setminus Q^\lambda}|K(y,z)|\,d\mu(z).
\end{align*}
For $y\in Q$ and $AB_Q\setminus Q^\lambda$, we have
$$|K(y,z)|\lesssim \frac1{|y-z|^n}\lesssim \frac1{(\lambda\,\ell(Q))^n}.$$
Thus,
$$|\langle T_\mu\chi_{AB_Q\setminus Q^\lambda} \rangle_S| + |\langle T_\mu\chi_{AB_Q\setminus Q^\lambda} \rangle_Q|\lesssim \frac{\mu(AB_Q)}{(\lambda\,\ell(Q))^n} \lesssim A^n\lambda^{-n}\,\Theta_\mu(AB_Q)
\leq A^n\lambda^{-n}\,\delta.$$

Concerning the term $\langle T_\mu\chi_{Q} \rangle_Q$, by Fubini, we have
$$\langle T_\mu\chi_{Q} \rangle_Q = \frac1{\mu(Q)} \iint_{Q\times Q} K(y,z)\,d\mu(y)\,d\mu(z) = \frac1{\mu(Q)}\iint_{Q\times Q} K(z,y)\,d\mu(y)\,d\mu(z),$$
and so
$$\langle T_\mu\chi_{Q} \rangle_Q = \frac1{2\,\mu(Q)} \iint_{Q\times Q} (K(y,z) + K(z,y))\,d\mu(y)\,d\mu(z).$$
Observe now that, by the antisymmetry of $\nabla_1 \Theta(\cdot,\cdot;A(y))$,
\begin{align*}
|K(y,z) + K(z,y)| &\leq |\nabla_1 \mathcal{E}_A (y,z) - \nabla_1   \Theta(y,z;A(y))| +  |\nabla_1  \mathcal{E}_A (z,y) + \nabla_1  \Theta(y,z;A(y))|\\
& = |\nabla_1  \mathcal{E}_A (y,z) - \nabla_1  \Theta(y,z;A(y))| +  |\nabla_1   \mathcal{E}_A (z,y) -\nabla_1   \Theta(z,y;A(y))|,
\end{align*}
and then, by \eqref{eqap*1} (which follows from Lemma \ref{ENV.constantCoef}), we have
$$\int_Q |K(y,z) + K(z,y)|\,d\mu(z)\lesssim \ell(Q)^\alpha,$$
for all $y\in Q$.
Therefore,
\begin{equation}\label{eq17}
|\langle T_\mu\chi_{Q} \rangle_Q |\lesssim \ell(Q)^\alpha.
\end{equation}

From \eqref{eqsplit22} and the preceding estimates we derive
$$\bigl|\Delta_Q T\mu(x) -  \langle T_\mu\chi_{Q} \rangle_S\bigr| \lesssim  
|\langle T_\mu\chi_{Q^\lambda\setminus Q} \rangle_S| + |\langle T_\mu\chi_{Q^\lambda\setminus Q} \rangle_Q|+
A^{-\gamma} +
A^n\lambda^{-n}\,\delta + \ell(Q)^\alpha,$$
for all $x\in S$.
 Integrating with respect to the measure $\mu|_Q$ and using Cauchy-Schwarz inequality, the latter estimate implies
$$\Bigl\|\Delta_Q T \mu - \sum_{S\in\Sigma_1(Q)} \langle T_\mu\chi_Q\rangle_S\,\chi_S
\Bigr\|_{L^2(\mu)}^2\lesssim \| T_\mu\chi_{Q^\lambda\setminus Q}\|_{L^2(\mu)}^2 + 
(A^{-\gamma} +
A^n\lambda^{-n}\,\delta + \ell(Q)^\alpha)^2\,\mu(Q).$$
Using the $L^2(\mu)$ boundedness of $T_\mu$ and \eqref{eqthin1}, we derive
$$\| T_\mu\chi_{Q^\lambda\setminus Q}\|_{L^2(\mu)}^2\lesssim \mu(Q^\lambda\setminus Q)\lesssim
\lambda^{1/2}\,\mu(Q).$$
Therefore,
$$\Bigl\|\Delta_Q T \mu - \sum_{S\in\Sigma_1(Q)} \langle T_\mu\chi_Q\rangle_S\,\chi_S
\Bigr\|_{L^2(\mu)}^2\lesssim 
(A^{-\gamma} +
A^n\lambda^{-n}\,\delta + \ell(Q)^\alpha + \lambda^{1/4})^2\,\mu(Q).$$
So the lemma follows if we take $A$ big enough, $\ell(Q)$ small enough, $\lambda$ small enough, and finally, $\delta$ small enough, depending on $A$ and $\lambda$.
\end{proof}

\vv
Recall that we denoted $\eta=\tilde{\chi}_Q\,\mu$. For a function $g$ and $S\in\DD$, we will write
$$\langle g\rangle_{\eta,S}=\frac1{\eta(S)}\int_S g\,d\eta.$$
We will also write $\langle g\rangle_{\mu,S}\equiv\langle g\rangle_{S}$ to emphasize the dependence on $\mu$ of
the latter notation.

\vv
\begin{lemma}\label{lemstep2}
We have
$$\Bigl\|\sum_{S\in\Sigma_1'(Q)} \langle T\eta\rangle_{\eta,S}\,\chi_S
\Bigr\|_{L^2(\eta)}^2\leq 4
\Bigl\|\sum_{S\in\Sigma_1'(Q)} \langle T_\mu\chi_Q\rangle_{\mu,S}\,\chi_S\Bigr\|_{L^2(\mu)}^2 + c\,(\ve_0 + \kappa_0^{1/3})\,\mu(Q).$$
\end{lemma}

\begin{proof}
By the $L^2(\mu)$ boundedness of $T_\mu$, we have
\begin{align*}
\Bigl\|&\sum_{S\in\Sigma_1'(Q)} \langle T_\mu\chi_Q\rangle_{\mu,S}\,\chi_S
- \sum_{S\in\Sigma_1'(Q)} \langle T_\mu \tilde{\chi}_Q\rangle_{\mu,S}\,\chi_S \Bigr\|_{L^2(\mu)}^2\\
 & =\Bigl\| \sum_{S\in\Sigma_1'(Q)} \langle T_\mu (\chi_Q-\tilde{\chi}_Q)\rangle_{\mu,S}\,\chi_S \Bigr\|_{L^2(\mu)}^2 \lesssim \|\tilde{\chi}_Q-\chi_Q\|_{L^2(\mu)}^2 \lesssim (\ve_0 + \kappa_0^{1/2})\,\mu(Q).
\end{align*}
Therefore,
\begin{align}\label{eqcomb2}
\frac12\Bigl\|\sum_{S\in\Sigma_1'(Q)} \langle T_\mu \tilde{\chi}_Q\rangle_{\mu,S}&\,\chi_S
\Bigr\|_{L^2(\mu)}^2 \leq \Bigl\|\sum_{S\in\Sigma_1'(Q)} \langle T_\mu\chi_Q\rangle_{\mu,S}\,\chi_S\Bigr\|_{L^2(\mu)}^2\\
+
\Bigl\|\sum_{S\in\Sigma_1'(Q)}& \langle T_\mu\chi_Q\rangle_{\mu,S}\,\chi_S- \sum_{S\in\Sigma_1'(Q)} \langle T_\mu \tilde{\chi}_Q\rangle_{\mu,S}\,\chi_S \Bigr\|_{L^2(\mu)}^2\nonumber\\
&\leq \Bigl\|\sum_{S\in\Sigma_1'(Q)} \langle T_\mu\chi_Q\rangle_{\mu,S}\,\chi_S\Bigr\|_{L^2(\mu)}^2+ c\,(\ve_0 + \kappa_0^{1/2})\,\mu(Q).\nonumber
\end{align}

Next we show that
\begin{align}\label{eqcomb3}
\Bigl\|\sum_{S\in\Sigma_1'(Q)} \langle T_\mu \tilde{\chi}_Q\rangle_{\mu,S}\,\chi_S -
\sum_{S\in\Sigma_1'(Q)} \langle T_\mu \tilde{\chi}_Q\rangle_{\eta,S}\,\chi_S
\Bigr\|_{L^2(\mu)}^2 & \lesssim 
\kappa_0^{1/3}\,\mu(Q).
\end{align}
To this end, note that for any function $g\in L^2(\mu)$, writing $I(S)=I_{\kappa_0}(S)$ to shorten notation,
\begin{align*}
\bigl|\langle g\rangle_{\mu,S} - \langle g\rangle_{\eta,S}\bigr| & =
\left|\frac1{\mu(S)} \int_S g\,d\mu - \frac1{\mu(I(S))} \int_{I(S)} g\,d\mu\right|\\
& \leq
\frac{\mu(S\setminus I(S))}{\mu(S)\,\mu(I(S))} \int_S |g|\,d\mu
+ \frac1{\mu(I(S))} \int_{S\setminus I(S)} |g|\,d\mu.
\end{align*}
To estimate the first term on the right hand side we use that 
$$\mu(S\setminus I(S))
\lesssim \kappa_0^{1/2}\,\mu(S)\lesssim \kappa_0^{1/2}\,\mu(I(S)).$$ For the second one, by H\"older's
inequality, for any $1<p<\infty$,
$$\frac1{\mu(I(S))} \int_{S\setminus I(S)} |g|\,d\mu \leq \left(\frac{\mu(S\setminus I(S))}{\mu(I(S))} 
\right)^{1/p'}\,\langle |g|^p\rangle_{\mu,S}^{1/p}\lesssim \kappa_0^{1/(2p')}\,\langle |g|^p\rangle_{\mu,S}^{1/p}.
$$
So we get
$$\bigl|\langle g\rangle_{\mu,S} - \langle g\rangle_{\eta,S}\bigr|\lesssim 
\kappa_0^{1/2}\,\langle |g|\rangle_{\mu,S} 
+ \kappa_0^{1/(2p')}\,\langle |g|^p\rangle_{\mu,S}^{1/p}
\leq 2 \kappa_0^{1/(2p')}\,\langle |g|^p\rangle_{\mu,S}^{1/p}.
$$
Therefore, choosing $p\in(1,2)$,
\begin{align*}
\Bigl\|\sum_{S\in\Sigma_1'(Q)} \langle g\rangle_{\mu,S}\,\chi_S -
\sum_{S\in\Sigma_1'(Q)} \langle g\rangle_{\eta,S}\,\chi_S
\Bigr\|_{L^2(\mu)}^2  & =
\sum_{S\in\Sigma_1'(Q)} |\langle g\rangle_{\mu,S} - \langle g\rangle_{\eta,S}|^2\,\mu(S)\\
& \lesssim \kappa_0^{1/p'}\,\sum_{S\in\Sigma_1'(Q)} \langle |g|^p\rangle_{\mu,S}^{2/p}\,\mu(S)\\
&\leq \kappa_0^{1/p'}\,\| |g|^p\|_{L^{2/p}(\mu)}^{2/p} = \kappa_0^{1/p'}\,\| g\|_{L^{2}(\mu)}^{2}.
\end{align*}

Applying the preceding estimate to $g=T_\mu \tilde{\chi}_Q$ and $p=3/2$, and using the $L^2(\mu)$ boundedness of $T_\mu$, we obtain \eqref{eqcomb3}.
In combination with \eqref{eqcomb2}, this gives
\begin{multline*}
\Bigl\|\sum_{S\in\Sigma_1'(Q)} \langle T\eta\rangle_{\eta,S}\,\chi_S 
\Bigr\|_{L^2(\eta)}^2
 \leq
2\Bigl\|\sum_{S\in\Sigma_1'(Q)} \langle T_\mu \tilde{\chi}_Q\rangle_{\mu,S}\,\chi_S 
\Bigr\|_{L^2(\mu)} \\
\quad + 
2\Bigl\|\sum_{S\in\Sigma_1'(Q)} \langle T_\mu \tilde{\chi}_Q\rangle_{\mu,S}\,\chi_S -
\sum_{S\in\Sigma_1'(Q)} \langle T_\mu \tilde{\chi}_Q\rangle_{\eta,S}\,\chi_S
\Bigr\|_{L^2(\mu)}^2\\
 \leq 4\Bigl\|\sum_{S\in\Sigma_1'(Q)} \langle T_\mu\chi_Q\rangle_{\mu,S}\,\chi_S\Bigr\|_{L^2(\mu)}^2
+  c\,(\ve_0 + \kappa_0^{1/2})\,\mu(Q) +
c\,\kappa_0^{1/3}\,\mu(Q),
\end{multline*}
which yields the desired estimate.
\end{proof}
\vv

\begin{lemma}\label{lemstep3}
For any $\ve>0$, if $Q\in\Sigma_N$ and $N$ is big enough, 
$A$ is big enough, and $\delta$ small enough (depending also on $A$ and $\kappa_0$), then
$$\|T\sigma\|_{L^2(\sigma)}^2\lesssim\Bigl\|\sum_{S\in\Sigma_1'(Q)} \langle T\eta\rangle_{\eta,S}\,\chi_S
\Bigr\|_{L^2(\eta)}^2 + \ve\,\mu(Q).$$
\end{lemma}

\begin{proof}
For all $x\in\tfrac14B(S)$, $S\in\Sigma_1'(Q)$, we write
\begin{align}\label{eqt123*}
\bigl| T\sigma(x)\bigr|  &\leq \bigl|T(\chi_{\tfrac14B(S)}\sigma)(x)\bigr| \\
&+  \bigl|T(\chi_{\R^{n+1}\setminus \tfrac14B  (S)}\sigma)(x) - T(\chi_{\R^{n+1}\setminus S}\,\eta)(x)\bigr|\nonumber\\
 &+ \bigl|T(\chi_{\R^{n+1}\setminus S}\,\eta)(x)- \langle T\eta\rangle_{\eta,S}\bigr|
+\bigl|\langle T\eta\rangle_{\eta,S}\bigr|=: T_1 + T_2 + T_3 + \bigl|\langle T\eta\rangle_{\eta,S}\bigr|.\nonumber
\end{align}

Using that $$\sigma|_{\tfrac14B(S)} = \mu(I(S))\,\frac{\LL^{n+1}|_{\tfrac14B  (S)}}{\LL^{n+1}\bigl(\tfrac14B  (S)\bigr)},$$
it follows that
$$T_1\lesssim  \frac{\mu(I(S))}{r(B(S))^{n+1}} \int_{\frac14B(S)} \frac1{|x-y|^n}\,d\LL^{n+1}(y)
\lesssim \frac{\mu(S)}{r(B(S))^n}\lesssim A^n \delta,$$
 where in the last inequality we used that $S$ has low density. 

Next we will deal with the term $T_3$ in \eqref{eqt123*}.
To this end, for $x\in\tfrac14B(S)$  we have
\begin{align}\label{eqsd539}
\bigl|T(\chi_{\R^{n+1}\setminus S}\,\eta)(x)&- \langle T\eta\rangle_{\eta,S}\bigr| \leq \bigl|T(\chi_{\R^{n+1}\setminus 2B_S}\,\eta)(x) - \langle T(\chi_{\R^{n+1}\setminus 2B_S}\,\eta)\rangle_{\eta,S}\bigr|\\
&+\bigl|T(\chi_{2B_S\setminus S}\,\eta)(x)\bigr| +  \bigl|\langle T(\chi_{2B_S\setminus S}\,\eta)\rangle_{\eta,S}\bigr| + \bigl|\langle T(\chi_{S}\,\eta)\rangle_{\eta,S}\bigr|.\nonumber
\end{align} 
The second term on the right hand side of \eqref{eqsd539} satisfies
\begin{equation*}
\bigl|T(\chi_{2B_S\setminus S}\,\eta)(x)\bigr|\leq \int_{2B_S\setminus S}\frac1{|x-y|^n}\,d\eta(y)
\lesssim \frac{\eta(2B_S)}{r(B(S))^n}\lesssim A^n\, \delta,
\end{equation*}
recalling that $x\in \frac14B(S)$ and that $\Theta_\eta(2B_S)\lesssim A^n \,\delta$ for the last estimate.

A similar argument works for the term $ \bigl|\langle T(\chi_{2B_S\setminus S}\,\eta)\rangle_{\eta,S}\bigr|$.
The main difference is that in the case when $x\in I(S)$ and $y\in 2B_S\setminus S$, we can only ensure that $|x-y|\geq \kappa_0\,\ell(S)$, and thus we derive
$$\bigl|T(\chi_{2B_S\setminus S}\,\eta)(x)\bigr|\leq \int_{2B_S\setminus S}\frac1{|x-y|^n}\,d\eta(y)
\lesssim \frac{\eta(2B_S)}{\kappa_0^n\,r(B(S))^n}\lesssim \kappa_0^{-n} A^n\,\delta.$$
Averaging over $x\in I(S)$  with respect to $\mu$ and recalling that $\mu(I(S))=\eta(S)$, we obtain
$$ \bigl|\langle T(\chi_{2B_S\setminus S}\,\eta)\rangle_{\eta,S}\bigr|\lesssim \kappa_0^{-n}A^n\,\delta.$$

Now we turn our attention to the first term on the right hand side of \eqref{eqsd539}. For $x'\in S$, 
we have
\begin{align}\label{eqigu3*}
\bigl|T(\chi_{\R^{n+1}\setminus 2B_S}\,\eta)(x) &-T(\chi_{\R^{n+1}\setminus 2B_S}\,\eta)(x')\bigr|\\
&\leq \int_{\R^{n+1}\setminus 2B_S} \bigl|K(x-y)-K(x'-y)\bigr|\,d\eta(y)\nonumber\\
& \lesssim \left(\int_{AB_S\setminus 2B_S} +
\int_{\R^{n+1}\setminus AB_S}\right) \frac{\ell(S)^\gamma}{|y-x_S|^{n+\gamma}}
\,d\eta(y)\nonumber\\
& \lesssim \frac{\mu(AB_S)}{\ell(S)^n} + \frac{c_0}{A^\gamma}\lesssim \delta\,A^{n} + A^{-\gamma},
\nonumber
\end{align}
taking into account that $S\in {\rm LD}(R)$ for some $R\in\DD$ and $\mu$ has polynomial growth with constant $c_0$ for the last inequality. 
Averaging on $x'\in S$ with respect to $\eta$ we get
$$\bigl|T(\chi_{\R^{n+1}\setminus 2B_S}\,\eta)(x) - \langle T(\chi_{\R^{n+1}\setminus 2B_S}\,\eta)\rangle_{\eta,S}\bigr|
\lesssim \delta\,A^{n} + A^{-\gamma}.$$

Regarding the term $\bigl|\langle T(\chi_{S}\eta)\rangle_{\eta,S}\bigr|$, arguing exactly as in
\eqref{eq17}, we obtain
$$\bigl|\langle T(\chi_{S}\,\eta)\rangle_{\eta,S}\bigr|\lesssim \ell(S)^\alpha\lesssim \ell(Q)^\alpha.$$
Thus, we derive
\begin{equation}\label{eq1934}
T_3=
\bigl|T(\chi_{\R^{n+1}\setminus S}\,\eta)(x)- \langle T\eta\rangle_{\eta,S}\bigr|\lesssim \kappa_0^{-n}\,A^n\,\delta + 
\delta\,A^{n} + A^{-\gamma} + \ell(Q)^\alpha =: \ve_1.
\end{equation}

Finally we deal with the term $T_2$.
For $x\in\tfrac14B  (S)$ we write
\begin{align}\label{eqt2***}
T_2 &= \bigl|T(\chi_{\R^{n+1}\setminus \tfrac14B  (S)}\sigma)(x) - T(\chi_{\R^{n+1}\setminus S}\,\eta)(x)\bigr| \\
&\leq \sum_{R\in \Sigma_1'(Q)\setminus\{S\}} \left|\int K(x-y)\,d(\sigma|_{\frac14B  (R)} -  \eta|_R)\right|\nonumber\\
&\leq
\sum_{R\in \Sigma_1'(Q)\setminus\{S\}} \int |K(x-y) - K(x-z_R)|\,d(\sigma|_{\frac14B  (R)} +  \eta|_R),\nonumber
\end{align}
using that $\sigma(\tfrac14B  (R)) =  \eta(R)$ for the last inequality.

Denote
$$D(R,S)=\ell(R) + \ell(S) + \dist(R,S).$$
From the fact that  $\frac12B(R)\cap \frac12B(S)=\varnothing$ and that 
$\supp( \eta|_R) = I_{\kappa_0}(R)\subset R$ for all $R\in\Sigma_1'(Q)$, it follows easily that
if $x\in \tfrac14B  (S)$ and $y\in\tfrac14B  (R)\cup \supp( \eta|_R)$, then
\begin{equation*}
|x-y|\gtrsim \kappa_0\,D(R,S).
\end{equation*}
We leave the details for the reader. Then, using the property (b) of the kernel $K$ in Lemma \ref{lemcz}, one gets
\begin{equation}\label{eqdqr1}
|K(x-y) - K(x-z_R)|\lesssim c(\kappa_0)\,\frac{\ell(R)^\gamma}{D(R,S)^{n+\gamma}}.
\end{equation}
Plugging this estimate into \eqref{eqt2***}, we obtain
$$T_2\lesssim c(\kappa_0)
\sum_{R\in \Sigma_1'(Q)\setminus\{S\}} \frac{\ell(R)^\gamma\, \eta(R)}{D(R,S)^{n+\gamma}}.$$

So from \eqref{eqt123*} and the estimates for the terms $T_1,T_2,T_3$, we infer that
for all $x\in\tfrac14B  (S)$ with $S\in\Sigma_1'(Q)$,
\begin{equation}\label{eqnos58}
\bigl|T\sigma(x)\bigr|  \lesssim  \bigl|\langle T\eta\rangle_{\eta,S}\bigr|+
c(\kappa_0)
\sum_{R\in \Sigma_1'(Q)\setminus\{S\}} \frac{\ell(R)^\gamma\, \eta(R)}{D(R,S)^{n+\gamma}}+ \ve_1,
\end{equation}
where $\ve_1$ is defined in \eqref{eq1934}.
Denote 
$$
\wt g(x) = \sum_{S\in \Sigma_1'(Q)}\sum_{R\in \Sigma_1'(Q)\setminus\{S\}} \frac{\ell(R)^\gamma\, \eta(R)}{D(R,S)^{n+\gamma}}\,\chi_{\frac14B(S)}(x).$$ 
Squaring and integrating \eqref{eqnos58} with respect to $\sigma$, we get
\begin{align}\label{eqff593}
\bigl\|T\sigma\bigr\|_{L^2(\sigma)}^2  & \lesssim  
\sum_{S\in \Sigma_1'(Q)} \bigl|\langle T\eta\rangle_{\eta,S}\bigr|^2\,\sigma(\tfrac14B(S)) +
\ve_1^2\,\sigma(Q) +  c(\kappa_0)^2\,\|\wt g\|_{L^2(\sigma)}^2,
\end{align}
Note that, since $\sigma(\tfrac14B(S))=\eta(S)$ for each $S$, 
\begin{align}\label{eqff594}
\sum_{S\in \Sigma_1'(Q)} \bigl|\langle T\eta\rangle_{\eta,S}\bigr|^2\,\sigma(\tfrac14B(S)) 
=\Bigl\|\sum_{S\in\Sigma_1'(Q)} \langle T\eta\rangle_{\eta,S}\,\chi_S
\Bigr\|_{L^2(\eta)}^2.
\end{align}
 By the same reasoning we 
deduce that $\|\wt g\|_{L^2(\sigma)}^2 = \| g\|_{L^2(\eta)}^2$, where
$$
 g(x) = 
 \sum_{S\in \Sigma_1'(Q)}
 \sum_{R\in \Sigma_1'(Q)\setminus\{S\}} \frac{\ell(R)^\gamma\, \eta(R)}{D(R,S)^{n+\gamma}}\,\chi_{S}(x).
$$

We will estimate $\|g\|_{L^2(\eta)}$ by duality: for any non-negative function $h\in L^2(\eta)$,
we get
\begin{align}\label{eqgh57}
\int g\,h\,d\eta &= 
\sum_{S\in \Sigma_1'(Q)}
 \sum_{R\in \Sigma_1'(Q)\setminus\{S\}} \frac{\ell(R)^\gamma\, \eta(R)}{D(R,S)^{n+\gamma}} \,\int_S h\,d\eta\\
&= \sum_{R\in \Sigma_1'(Q)}  \eta(R)
 \sum_{S\in \Sigma_1'(Q)\setminus\{R\}} \frac{\ell(R)^\gamma}{D(R,S)^{n+\gamma}} \,\int_S h\,d\eta.\notag
 \end{align}
 For each $z\in R\in \Sigma_1'(Q)$ we have 
\begin{align*}
 \sum_{S\in \Sigma_1'(Q)\setminus\{R\}} \frac{\ell(R)^\gamma}{D(R,S)^{n+\gamma}} \,\int_S h\,d\eta &\lesssim
\int \frac{\ell(R)^\gamma\,h(y)}{\bigl(\ell(R)+|z-y|\bigr)^{n+\gamma}}\,d\eta(y)\\
& = \int_{|z-y|\leq \ell(R)}\cdots + \sum_{j\geq1}\int_{2^{j-1}\ell(R)<|z-y|\leq 2^{j}\ell(R)}\cdots \\
& \lesssim 
\sum_{j\geq0} \langle h\rangle_{\eta,B(z,2^{j}\ell(R))} \;\,\frac{2^{-j\gamma}\,\eta(B(z,2^j\ell(R)))}{\bigl(2^j\ell(R)
\bigr)^n}.
\end{align*}
Now we take into account that $\langle h\rangle_{\eta,B(z,2^{j}\ell(R))}\lesssim
M_{\eta}h(z)$, where $M_{\eta}$ stands for the centered maximal Hardy-Littlewood operator with respect to $\eta$, and that
$$
\sum_{j\geq0}\frac{2^{-j\gamma}\,\eta(B(z,2^j\ell(R)))}{\bigl(2^j\ell(R)\bigr)^n}\lesssim
\delta\,A^{n} + A^{-\gamma}\leq \ve_1,
$$
by arguments
analogous to the ones in \eqref{eqigu3*}. Then, by \eqref{eqgh57},
\begin{align*}
\int g\,h\,d\eta &\lesssim
\ve_1 \sum_{R\in \Sigma_1'(Q)} \inf_{z\in R} M_{\eta}h(z)\,
 \eta(R) \leq \ve_1\int M_{\eta}h\, \,d\eta \lesssim \ve_1\,\|h\|_{L^2(\eta)}\,\eta(Q)^{1/2},
 \end{align*}
which implies that $
\|g\|_{L^2(\eta)}\lesssim \ve_1\,\eta(Q)^{1/2}$.
Plugging this estimate into \eqref{eqff593} and recalling that $\|\wt g\|_{L^2(\sigma)} = \| g\|_{L^2(\eta)}$, we obtain
$$\bigl\|T\sigma\bigr\|_{L^2(\sigma)}^2   \lesssim  
\sum_{S\in \Sigma_1'(Q)} \bigl|\langle T\eta\rangle_{\eta,S}\bigr|^2\,\sigma(\tfrac14B(S)) +
c'(\kappa_0)\ve_1^2\,\sigma(Q),$$
which by \eqref{eqff594} proves the lemma.
\end{proof}
\vv

Notice that Lemma \ref{ENV.approximation} is an immediate consequence of Lemmas 
\ref{lemstep1}, \ref{lemstep2}, and \ref{lemstep3}.

\vv

\section{Proof of Proposition \ref{ENV.mainlemma} using a variational argument}

 With Lemma \ref{ENV.approximation} at our disposal, the proof of Proposition \ref{ENV.mainlemma} can readily be concluded once we estimate $\|T\sigma\|_{L^2(\sigma)}$ from below. 
\vv

Given a fixed cube $Q \in \Sigma$ and $0<\lambda<1$, for the
sake of contradiction, we assume that
\begin{equation}\label{eqassu8}
\|T\sigma\|_{L^2(\sigma)}^2 \leq \lambda \|\sigma\|.
\end{equation}
We will show that $\lambda$ cannot be arbitrarily small for $Q \in \Sigma_N$, $N>N_0$ big enough.
\vv

\subsection{A pointwise estimate}

Consider the family of functions
$$\AZ:= \Bigl\{g\in  L^\infty(\sigma): g\geq0 \mbox{ and $\int g \,d\sigma = \|\sigma\|$} \Bigr\}$$
and the functional defined on $\AZ$ by
$$
F(g) = \lambda \|g\|_{\infty} \|\sigma\| + \int \left|T(g\sigma) \right|^2 g d\sigma.
$$
Notice that by hypothesis we have 
$$
F(\chi_Q) = \lambda \,\|\sigma\| + \int \left|T\sigma \right|^2 d\sigma \leq 2\lambda \|\sigma\|.
$$
Therefore, the infimum of the functional $F$ is attained over  functions $g$ that satisfy $\|g\|_{L^\infty(\sigma)} \leq 2$. Indeed, recall that $\sigma$ is absolutely continuous
with respect to Lebesgue measure and that its density function is bounded (because $\Sigma_1'(Q)$ is a finite family).  Then, by taking a weak limit of functions $b_k\in L^\infty(\sigma)$ such that $\|b_k\|_{L^\infty(\sigma)}\leq 2$ 
and $F(b_k)\to \inf_\AZ F$, it follows easily that 
$$\|b_k\|_{L^\infty(\sigma)}\to \|b\|_{L^\infty(\sigma)}
\quad\mbox{ and }\quad \int \left|T(b_k\sigma) \right|^2 \,b_k\,d\sigma\to \int \left|T(b\sigma) \right|^2 \,b\,d\sigma,$$
and thus $F(b_k)\to F(b)$ and the minimum of $F$ in $\AZ$ is attained at $b$, with 
 $\|b\|_{L^\infty(\sigma)} \leq 2$.

For the rest of this section we denote 
$$
d\nu = b\, d\sigma.
$$
Obviously,  since $b$ is a minimizer of $F$,
\begin{equation}\label{eqobv1}
\lambda \|b\|_{\infty} \|\sigma\| + \int \left|T\nu \right|^2  d\nu\leq 2\lambda \,\|\sigma\|=2\lambda \,\|\nu\|.
\end{equation}

Now fix some point $x \in \mathrm{supp}(\nu)$ and some small radius $r$.
Denote $B=B(x,r)$ and consider the following variation of the minimizer $b$:
$$
b_t = b\,(1-t \chi_B) + t\, b\, \frac{\nu(B)}{\|\nu\|}.
$$
Denote also by $\nu_t$ the measure given by $d\nu_t  = b_t \,d\sigma$. Compute
\begin{align*}
F(b_t) & =  \lambda \|b_t\|_\infty \|\nu\| + \int |T\nu_t|^2 d\nu_t \\
& \leq  \lambda \|b\|_\infty \left( 1 + t\,\frac{\nu(B)}{\|\nu\|} \right)\|\nu\| + \int |T\nu_t|^2 d\nu_t =: G(t). 
\end{align*}

Since $b$ is a minimizer for the functional $F$ and $G(0) = F(b)$ we have the chain
$$
G(0) = F(b)\leq F(b_t) \leq G(t),
$$
which implies that $G'(0_+) := \lim_{t \to 0+} G'(t) \geq 0$. This means that
$$
\lambda \|b\|_\infty \nu(B) + \int |T\nu|^2 \left(\frac{\nu(B)}{\|\nu\|}  - \chi_B \right) d\nu + 
2 \int T\nu \cdot T\left(\left[\frac{\nu(B)}{\|\nu\|} - \chi_B\right]\nu \right) d\nu \geq 0 .
$$
Reorganizing, we get
\begin{equation}\label{eq1p22}
\int_B |T\nu|^2 d\nu + 2\int T\nu \cdot T(\chi_B \nu) d\nu \leq \lambda \|b\|_\infty \nu(B) + \frac{\nu(B)}{\|\nu\|} \int |T\nu|^2\, d\nu + 2 \,\frac{\nu(B)}{\|\nu\|} \int |T\nu|^2 \,d \nu.
\end{equation}
Given a vector valued measure $\omega$, we define
$$T^* \omega(x)  =  \int \nabla_1 \mathcal{E} (y,x) \cdot d\omega(y).$$
Now we take into account  that 
$$
\int T\nu\cdot T(\chi_B \nu) d\nu  = \int_B T^*\left( [T\nu]\nu \right) d\nu,
$$
we divide both sides of \eqref{eq1p22} by $\nu(B)$, and we use \eqref{eqobv1} to get
\begin{align*}
\frac{1}{\nu(B)} \int_B |T\nu|^2 d\nu + \frac{2}{\nu(B)} \int_B T^*\left( [T\nu]\nu \right) d\nu & \leq \lambda \|b\|_\infty + \frac{3}{\|\nu\|} \int |T\nu|^2 d\nu \leq 6\lambda. \\
\end{align*}

Finally, we let $r=r(B)$ tend to $0$ and Lebesgue differentiation theorem yields the pointwise inequality
\begin{equation}\label{ENV.pointwiseineqnu}
|T\nu(x)|^2 + 2 T^*\left( [T\nu]\nu \right)(x) \leq 6 \lambda, \quad \mbox{for} \; \nu\mbox{-a.e.} \; x \in \mathrm{supp}(\nu).
\end{equation}

\vspace{2mm}
\subsection{Application of the maximum principle}

We now want to extend \eqref{ENV.pointwiseineqnu} to the whole of $\mathbb{R}^{n+1}$. 
Note first that $T\nu$ is continuous and belongs to $L^\infty(\nu)$ because 
$b\in L^\infty(\sigma)$, and 
 $\sigma$ is absolutely continuous
with respect to Lebesgue measure with a bounded density function, since $\Sigma_1'(Q)$ is a finite family. Hence $|T\nu(x)|^2$ is continuous, and also 
$T^*\left( [T\nu]\nu \right)$ by analogous arguments.

Now we claim that the definition of $T$ implies
\begin{equation}\label{ENV.simplemaxprincestimate}
\sup_{x \in \mathbb{R}^{n+1}} |T^*\omega(x)| \leq \sup_{x \in \mathrm{supp}(\omega)} |T^*\omega(x)|,
\end{equation}
for each vector valued measure $\omega$ which is compactly supported and absolutely continuous with respect to Lebesgue measure with a bounded density function. 
Indeed, if $d\omega = \vec{F} d\LL^{n+1}$, for $x\in \supp(\omega)$ we may write

$$
T^* \omega(x)  =  \int \nabla_1 \mathcal{E} (y,x) \cdot d\omega(y) =  \int \nabla_1 \mathcal{E} (y,x) \cdot \vec{F}(y) dy. 
$$

Note now that if $\vphi \in C^\infty_c(\R^{n+1} \setminus \supp\vec{F})$, by Fubini's theorem and the fact that $\EE_A(y,x)=\EE_{A^*}(x,y)$ (where $\EE_{A^*}(\cdot,\cdot)$ is the fundamental solution of
the operator defined by $L_{A^*} u(x)= -\mathrm{div}\left(A^*(\cdot) \nabla u (\cdot) \right)(x)$),
  we have that 
$$
\int A^* \nabla T^*\omega \cdot \nabla \vphi =  \int \nabla_y \!\!  \int A^*(x) \nabla_x \mathcal{E}_{A^*}(x,y) \cdot \nabla \vphi(x) \, dx \,\cdot  \vec{F}(y) dy = \int  \nabla \vphi \cdot \vec{F} = 0.
$$ 
Therefore, $T^*\omega$ is $L^*$-harmonic away from the support of $\omega$, and by maximum principle we get \eqref{ENV.simplemaxprincestimate}.

To derive the desired extension of \eqref{ENV.pointwiseineqnu} from \eqref{ENV.simplemaxprincestimate}, we use the elementary formula
$$
\frac12 |w|^2 = \sup_{\begin{subarray}{c} \beta > 0 \\ e \in \mathbb{R}^{n+1}, \|e\|=1 \end{subarray}} \beta \,\langle e, w \rangle - \frac12 \beta^2.
$$ 
We apply it for  $w=T\nu(x)$, with $x\in\R^{n+1}$, and we get
\begin{equation}
\label{ENV.Tnucuadrado}
\frac12 |T\nu(x)|^2 = \sup_{\begin{subarray}{c} \beta > 0 \\ e \in \mathbb{R}^{n+1}, \|e\|=1 \end{subarray}} \beta \,\langle e, T\nu(x) \rangle - \frac12 \beta^2.
\end{equation}
Now, if $e=(e_1, \cdots e_{n+1})$ and we define the vector valued measure $\nu e = (\nu e_1, \cdots, \nu e_{n+1})$, taking into account that $\nabla_1 \Theta(y,x;A(x)) =- \nabla_1 \Theta(x,y;A(x)) $  for all $x\neq y$, we obtain
\begin{align*}
\langle e, T\nu(x) \rangle & = \int \nabla_1 \mathcal{E}(x,y) \cdot e \; d\nu(y) \\
& =  \int \nabla_1 \mathcal{E}(x,y) \cdot d(\nu e)(y) \\
& =  -T^*(\nu e)(x) + \int \left[ \nabla_1 \mathcal{E}(x,y) + \nabla_1 \mathcal{E}(y,x) \right]\cdot d(\nu e)(y) \\
& =  -T^*(\nu e)(x) + \int \left[ \nabla_1 \mathcal{E}(x,y) - \nabla_1 \Theta(x,y;A(x)) \right]\cdot d(\nu e)(y) \\
& \quad+  \int \left[ \nabla_1 \mathcal{E}(y,x) - \nabla_1 \Theta(y,x;A(x)) \right]\cdot d(\nu e)(y).
\end{align*}
By Lemma \ref{ENV.constantCoef}, if $\dist(x,Q)\leq 1$,
 the last two terms above are small and we get
\begin{equation}\label{eqcas1*}
\langle e, T\nu(x) \rangle = -T^*(\nu e)(x) + e\cdot C(x)\ell(Q)^\alpha,
\end{equation}
with $|C(x)|\lesssim1$. In the case that $\dist(x,Q)\geq1$, we use the fact that 
$|\nabla_1 \mathcal{E}(x,y)|+ |\nabla_1 \mathcal{E}(y,x)|\lesssim1$ by Lemma \ref{lemcz} (c).
Then we derive
$$\left|\int \left[ \nabla_1 \mathcal{E}(x,y) + \nabla_1 \mathcal{E}(y,x) \right]\cdot d(\nu e)(y)\right|
\lesssim \|\nu\|\lesssim \mu(Q)\lesssim \ell(Q)^n\lesssim \ell(Q)^\alpha,$$
and so \eqref{eqcas1*} also holds.

We insert the above calculation in \eqref{ENV.Tnucuadrado} and we get 
\begin{align*}
|T\nu(x)|^2 &+ 4 T^*\left( [T\nu]\nu \right)(x) \\
& =  \sup_{\begin{subarray}{c} \beta > 0 \\ e \in \mathbb{R}^{n+1}, \|e\|=1 \end{subarray}} \left\{ -2\beta T^*(\nu e)(x) +e\cdot C(x) \beta \ell(Q)^\alpha -\beta^2 + 4 T^*\left( [T\nu]\nu \right)(x) \right\} \\
& =   \sup_{\begin{subarray}{c} \beta > 0 \\ e \in \mathbb{R}^{n+1}, \|e\|=1 \end{subarray}} \left\{ T^*\left(-2\beta \nu e + 4[T\nu]\nu \right)(x) + e\cdot C(x) \beta \ell(Q)^\alpha -\beta^2 \right\} \\
& \leq  \sup_{\begin{subarray}{c} \beta > 0 \\ e \in \mathbb{R}^{n+1}, \|e\|=1 \end{subarray}} \sup_{z \in \mathrm{supp}(\nu)} \left\{ T^*\left(-2\beta \nu e + 4[T\nu]\nu \right)(z) + e\cdot C(x) \beta \ell(Q)^\alpha -\beta^2 \right\} \\
& =   \sup_{z \in \mathrm{supp}(\nu)} \!\sup_{\begin{subarray}{c} \beta > 0 \\ e \in \mathbb{R}^{n+1}, \|e\|=1 \end{subarray}} \left\{ T^*\left(-2\beta \nu e + 4[T\nu]\nu \right)(z) + e\cdot C(x) \beta \ell(Q)^\alpha -\beta^2 \right\}, 
\end{align*}
by \eqref{ENV.simplemaxprincestimate}. 

Now we reverse the process admitting another error term which is bounded above by $\ell(Q)^\alpha$ to obtain
\begin{align*}
&|T\nu(x)|^2 + 4 T^*\left( [T\nu]\nu \right)(x)  \\&\leq   \sup_{z \in \mathrm{supp}(\nu)} \sup_{\begin{subarray}{c} \beta > 0 \\ e \in \mathbb{R}^{n+1}, \|e\|=1 \end{subarray}} \!\!\left\{ -2\beta T^*( \nu e)(z) + 4T^*\left( [T\nu]\nu \right)(z) + e\cdot C(x) \beta \ell(Q)^\alpha -\beta^2 \right\} \\
& = \sup_{z \in \mathrm{supp}(\nu)}\! \sup_{\begin{subarray}{c} \beta > 0 \\ e \in \mathbb{R}^{n+1}, \|e\|=1 \end{subarray}}  \! \!\!\left\{2\beta e\cdot T\nu(z) + e\cdot C'(z)\beta \ell(Q)^\alpha + 4T^*\left( [T\nu]\nu \right)(z) + e\cdot C(x) \beta \ell(Q)^\alpha -\beta^2  \right\} \\
& =  \sup_{z \in \mathrm{supp}(\nu)} \sup_{\begin{subarray}{c} \beta > 0 \\ e \in \mathbb{R}^{n+1}, \|e\|=1 \end{subarray}}\! \!\!\left\{2\beta  e\cdot(T\nu(z) - \frac12(C(x)+C'(z))\ell(Q)^\alpha + 4T^*\left( [T\nu]\nu \right)(z) -\beta^2  \right\} \\
& =  \sup_{z \in \mathrm{supp}(\nu)} \left\{ \left| T\nu(z) - \frac12(C(x)+C'(z))\ell(Q)^\alpha \right|^2 +4T^*\left( [T\nu]\nu \right)(z)\right\} \\
& \leq  \sup_{z \in \mathrm{supp}(\nu)} \left\{ 2|T\nu(z)|^2 + 4 T^*\left( [T\nu]\nu \right)(z)\right\} + C\ell(Q)^\alpha.  
\end{align*}
Finally, we apply \eqref{ENV.pointwiseineqnu} to get
\begin{equation}
\label{ENV.goodPointwise}
|T\nu(x)|^2 + 4\, T^*\left( [T\nu]\nu \right)(x) \lesssim \lambda + \ell(Q)^\alpha \quad \forall x \in \mathbb{R}^{n+1}.
\end{equation}

\vv

\subsection{The vector field $\Psi_Q$}

For each cube $R \in \mathrm{HD}(Q)$ consider a vectorial function $g_R$ such that
$$
T^*(g_R \; d\LL^{n+1}) = \varphi_R,
$$
where $\varphi_R$ is a smooth function such that $\chi_{1.5 B_R}\leq \varphi_R\leq \chi_{2 B_R}$,
with $\|\nabla\varphi_R\|_\infty\lesssim \ell(R)^{-1}$.
 By the reproducing formula
$$
\varphi(x) = \int \nabla_1 \mathcal{E}(y,x) \left[A^*(y)\nabla\varphi(y) \right] \; dy = T^*\left[A^*(\cdot)\nabla\varphi(\cdot) \; d\LL^{n+1}\right](x),
$$
which is valid for smooth functions with compact support, we may set 
$$
g_R := A^* \nabla\varphi_R.
$$
 Of course, we have
\begin{equation}\label{eqgr}
\supp g_R\subset 2B_R, \quad \|g_R\|_{L^\infty(\LL^{n+1})} \lesssim \ell(R)^{-1},\quad
\|g_R\|_{L^1(\LL^{n+1})} \lesssim \mu(R) \approx \ell(R)^n.
\end{equation}

Now, we define the following subcollection of cubes:
$$
{\rm HD}_0(Q):=\big\{R\in{\rm HD}(Q): \nu(1.5B_R)\geq \frac14\,\mu(R) \big\}.
$$
We have
$$\sum_{R\in {\rm HD}(Q)\setminus {\rm HD}_0(Q)}\nu(1.5B_R)\leq \frac14\,\sum_{R\in{\rm HD}(Q)}\mu(R) =\frac14\,\mu(Q)\leq \frac12\,\|\nu\|,$$
where in the last inequality we used \eqref{eq:sigma=muQ}. Thus, 
$$\|\nu\| \leq \sum_{R\in {\rm HD}_0(Q)}\nu(1.5B_R) + \sum_{R\in {\rm HD}(Q)\setminus {\rm HD}_0(Q)}\nu(1.5B_R)
\leq \sum_{R\in {\rm HD}_0(Q)}\nu(1.5B_R) + \frac12\,\|\nu\|,$$
and so 
$$\|\nu\| \leq 2\sum_{R\in {\rm HD}_0(Q)}\nu(1.5B_R).$$

Next, note  that for $R\in {\rm HD}_0(Q)$,  
$$\nu(1.5B_R)\geq \frac14\,\mu(R)\gtrsim\tau \ell(R)^n$$
 and observe that $\nu$ has
$n$-polynomial growth (this follows easily  from the $n$-polynomial growth of $\mu$). Thus,
$$\nu(9B_R)\lesssim_\tau \nu(1.5B_R).$$
Hence,
by a Vitali type covering argument we may find
a finite subfamily ${\rm HD}_1(Q)\subset {\rm HD}_0(Q)$ such that the balls $3B_R$ with $R\in{\rm HD}_1(Q)$
are pairwise disjoint, and such that
$$\mu(Q)\approx_\tau\|\nu\|\approx_\tau \sum_{R\in {\rm HD}_1(Q)}\nu(1.5B_R).$$
Now, we define
$$
\Psi_Q := \sum_{R \in \mathrm{HD}_1(Q)} g_R.
$$

We need the following auxiliary result:

\begin{lemma}\label{lemauxfi}
We have
\begin{equation}\label{eqap890}
\int |T(|\Psi_Q|\, \LL^{n+1})|^2 d\nu  \lesssim \mu(Q).
\end{equation}
\end{lemma}

\begin{proof}
First we will show that
\begin{equation}\label{eqap891}
\int_Q |T(|\Psi_Q|\, \LL^{n+1})|^2 d\mu\lesssim \mu(Q).
\end{equation}
To this end, for each $R\in{\rm HD}_1(R)$ we consider 
the function $h_R=c_R\,\chi_R$, with $c_R\in\R$ such that 
$$c_R\,\mu(R)= \int h_R\,d\mu = \int |g_R|\,d\LL^{n+1},$$
so that $0<c_R\lesssim 1$, by \eqref{eqgr}. We denote $h=\sum_{R\in{\rm HD}_1(R)} h_R$.
Then, for each $x\in\R^{n+1}$ we have
$$|T(|\Psi_Q|\, \LL^{n+1})(x)|\leq |T(|\Psi_Q|\, \LL^{n+1} - h\,\mu)(x)| +  |T(h\,\mu)(x)|.$$
Since $\int |T(h\mu)|^2 d\mu\lesssim \mu(Q)$, to prove \eqref{eqap891} it suffices to show that 
\begin{equation}\label{eqap892}
\int_Q |T(|\Psi_Q|\, \LL^{n+1} - h\,\mu)|^2\,d\mu\lesssim\mu(Q).
\end{equation}

For $x\in R\in {\rm HD}_1(Q)$, we write
\begin{align*}
|T(|\Psi_Q|\, \LL^{n+1} - h\,\mu)(x)| & \leq |T(|g_R|\, \LL^{n+1})(x)| + |T(h_R\,\mu)(x)|\\
&\quad  + \sum_{P\in{\rm HD}_1(Q)\setminus\{R\}} 
\left|\int \!K(x,y)\,(|g_P|\,d\LL^{n+1} - h_P\,d\mu)\right|.
\end{align*}
We have
\begin{equation}\label{eqpunt62}
|T(|g_R|\, \LL^{n+1})(x)|\lesssim \frac1{\ell(R)}\int_{2B_R} \frac1{|x-y|^n}\,d\LL^{n+1}(y)\lesssim 1.
\end{equation}
On the other hand, for each $P\in{\rm HD}_1(Q)\setminus\{R\}$ we have
\begin{align*}
\left|\int K(x,y)\,(|g_P|\,d\LL^{n+1} - h_P\,d\mu)\right| & \leq 
\int |K(x,y)-K(x,x_P)|\,(|g_P|\,d\LL^{n+1} + h_P\,d\mu)\\
& \lesssim \int \frac{\ell(P)^\gamma}{|x-x_P|^{n+\gamma}}\,(|g_P|\,d\LL^{n+1} + h_P\,d\mu)\\
& \lesssim \frac{\ell(P)^\gamma}{D(P,R)^{n+\gamma}}\,\mu(P),
\end{align*}
taking into account that the balls $3B_R$, $R\in{\rm HD_1}(Q)$, are disjoint.
Therefore,
\begin{align*}
\int_Q |T(|\Psi_Q|\, \LL^{n+1} - h\,\mu)|^2\,d\mu & \lesssim 
\mu(Q) + \sum_{R\in{\rm HD}_1(Q)} \int |T(h_R\,\mu)|^2\,d\mu \\
&\quad  + \sum_{R\in{\rm HD}_1(Q)}\biggl(\sum_{P\in{\rm HD}_1(Q)\setminus \{R\}} 
\frac{\ell(P)^\gamma}{D(P,R)^{n+\gamma}}\,\mu(P)\biggr)^2\,\mu(R).
\end{align*}
By the $L^2(\mu)$ boundedness of $T_\mu$, it is clear that $\int |T(h_R\,\mu)|^2\,d\mu\lesssim\mu(R)$. Also, arguing by duality, as in \eqref{eqgh57}, it follows easily that the last term on the right hand side is bounded above by $c\,\mu(Q)$.
So \eqref{eqap892} follows, and thus \eqref{eqap891} too.

Next we turn our attention to the integral in \eqref{eqap890}. For each $S\in \Sigma_1'(Q)$, denote by $R_S$ the cube from ${\rm HD}(Q)$ that contains $S$.
For all $x,y\in S\cup \frac14B(S)$, we write
\begin{align}\label{eqfs89}
|T(|\Psi_Q|\, \LL^{n+1})(x)- T(|\Psi_Q|\, &\LL^{n+1})(y)| \\
\leq \sum_{\substack{R\in{\rm HD}_1(Q):\\ 2B_{R_S}\cap 2B_R\neq\varnothing}}&
\bigl(|T(|g_R|\, \LL^{n+1})(x)| +|T(|g_R|\, \LL^{n+1})(y)|\bigr)\nonumber\\
&+
\sum_{\substack{R\in{\rm HD}_1(Q):\\ 2B_{R_S}\cap 2B_R=\varnothing}}
|T(|g_R|\, \LL^{n+1})(x)-T(|g_R|\, \LL^{n+1})(y)|.\nonumber
\end{align}
Since the balls $\frac12B(R)$, $R\in{\rm HD}(Q)$, are disjoint and have diameter comparable to $\ell(R_S)$, the first sum on the right hand side has a bounded number of summands. Then, by \eqref{eqpunt62}, we obtain
$$\sum_{\substack{R\in{\rm HD}_1(Q):\\ 2B_{R_S}\cap 2B_R\neq\varnothing}}
\bigl(|T(|g_R|\, \LL^{n+1})(x)| +|T(|g_R|\, \LL^{n+1})(y)|\lesssim 1.$$
The second sum on the right hand side of \eqref{eqfs89}, is bounded above by 
\begin{align*}
\sum_{\substack{R\in{\rm HD}_1(Q):\\ 2B_{R_S}\cap 2B_R=\varnothing}}
\int_{1.5B_R}\frac{\ell(S)^\gamma}{|x-z|^{n+\gamma}}\,|g_R(z)|\,d\LL^{n+1}(z)\lesssim
\sum_{\substack{R\in{\rm HD}_1(Q)}}
\frac{\ell(S)^\gamma}{D(R,R_S)^{n+\gamma}}\,\mu(R)\lesssim 1,
\end{align*}
taking into account that $|g_R|\lesssim \ell(R)^{-1}$ and
using the $n$-polynomial growth of $\mu$ in the last inequality.
So we infer that
$$
|T(|\Psi_Q|\, \LL^{n+1})(x)- T(|\Psi_Q|\, \LL^{n+1})(y)|\lesssim1\quad \mbox{ for all $x,y\in S\cup \frac14B(S)$.}$$
 
Using the last estimate, we get
\begin{align*}
\int |T(|\Psi_Q|\, \LL^{n+1})|^2 d\nu & \leq 2\sum_{S\in\Sigma_1'(Q)}\int_S |T(|\Psi_Q|\, \LL^{n+1})|^2 d\sigma\\
&\lesssim \mu(Q) + \sum_{S\in\Sigma_1'(Q)}\int_S \inf_{y\in S}|T(|\Psi_Q|\, \LL^{n+1})(y)|^2 d\sigma\\
& = \mu(Q) + \sum_{S\in\Sigma_1'(Q)}\int_S \inf_{y\in S}|T(|\Psi_Q|\, \LL^{n+1})(y)|^2 d\eta\\
& \lesssim \mu(Q) + \sum_{S\in\Sigma_1'(Q)}\int_S |T(|\Psi_Q|\, \LL^{n+1})|^2 d\eta\\
& \lesssim \mu(Q) + \int |T(|\Psi_Q|\, \LL^{n+1})|^2 d\mu\lesssim \mu(Q),
\end{align*}
which completes the proof of the lemma.
\end{proof}

\vv

\subsection{Contradiction}

We compute, by the construction of $\mathrm{HD}_1(Q)$,
\begin{align*}
\mu(Q) &\approx_\tau \sum_{R\in\mathrm{HD}_1(Q)} \nu(1.5B_R) \leq  \sum_{R\in\mathrm{HD}_1(Q)} \int \varphi_R\, d\nu \\
&=   \sum_{R\in\mathrm{HD}_1(Q)} \int T^*[g_R \; d\LL^{n+1}] d\nu  =  \int T^*[\Psi_Q \; d\LL^{n+1}] d\nu
= \int T\nu \cdot \Psi_Q \; d\LL^{n+1} \\
& \leq  \left(\int |T\nu|^2 |\Psi_Q| \; d\LL^{n+1}\right)^{\frac12} \left(\int |\Psi_Q| \; d\LL^{n+1}\right)^{\frac12}\\
& \leq  \left(\int |T\nu|^2 |\Psi_Q| \; d\LL^{n+1}\right)^{\frac12}\mu(Q)^{\frac12}. \\
\end{align*}
We now apply \eqref{ENV.goodPointwise} and we get
$$
\int |T\nu|^2 |\Psi_Q| \; d\LL^{n+1}  \lesssim \left(\lambda+ \ell(Q)^\alpha\right)\int |\Psi_Q| \; d\LL^{n+1} + \left|\int  T^*\left( [T\nu]\nu\right) |\Psi_Q| \; d\LL^{n+1} \right| =: \mathrm{I} + \mathrm{II}. 
$$
For $\mathrm{I}$, we know that
$$
\mathrm{I} \lesssim \left(\lambda+ \ell(Q)^\alpha\right) \sum_{R\in \mathrm{HD}_1(Q)} \mu(R) \leq \left(\lambda+ \ell(Q)^\alpha\right)\mu(Q). 
$$
On the other hand, concerning $\mathrm{II}$, by Cauchy-Schwarz, \eqref{eqobv1}, and Lemma \ref{lemauxfi} we get
\begin{align*}
\mathrm{II} & =  \left|\int T\nu \cdot T(|\Psi_Q| \, \LL^{n+1}) d\nu \right| \\
& \leq  \left(\int |T\nu|^2 d\nu \right)^{\frac12}\left(\int |T(|\Psi_Q| \, \LL^{n+1})|^2 d\nu \right)^{\frac12} \\
& \leq  \lambda^{\frac12} \mu(Q)^{\frac12} \mu(Q)^{\frac12} = \lambda^{\frac12} \mu(Q). 
\end{align*}
Finally, gathering all estimates in this subsection together yields
$$
\mu(Q) \lesssim _\tau\left(\lambda + \ell(Q)^\alpha\right)^{\frac14} \mu(Q),
$$
which is a contradiction if both $\ell(Q)$  and $\lambda$ are small enough. 

\vv

\end{document}